\documentclass[final]{siamltex1213}
\usepackage{epsf}
\usepackage{bm}
\usepackage{latexsym}
\usepackage{amsmath}
\usepackage{amssymb}
\usepackage{graphicx}
\usepackage{caption}
\usepackage{subcaption}
\usepackage{wrapfig}
\usepackage{color}
\usepackage{cite}
\usepackage{mathtools}
\usepackage{setspace}
\RequirePackage{amsopn}
\RequirePackage{amsfonts}

\mathtoolsset{showonlyrefs}

\title{\LARGE \bf From Weak Learning to Strong Learning in Fictitious Play Type Algorithms}
\author{BRIAN SWENSON$^{\dagger\star}$, SOUMMYA KAR$^\dagger$ AND JO\~{A}O XAVIER$^\star$\thanks{The work was partially supported by the FCT project FCT [UID/EEA/50009/2013] through the Carnegie-Mellon/Portugal Program managed by ICTI from FCT and by FCT Grant CMU-PT/SIA/0026/2009, and was partially supported by NSF grant ECCS-1306128. \newline$^\dagger$Department of Electrical and Computer Engineering, Carnegie Mellon University, Pittsburgh, PA 15213, USA (brianswe@andrew.cmu.edu and soummyak@andrew.cmu.edu).\newline $^\star$Institute for Systems and Robotics (ISR/IST), LARSyS, Instituto Superior T\'{e}cnico, Portugal (jxavier@isr.ist.utl.pt).}}

\newcommand\BRtildet{\bar{BR^{\epsilon_t}}}

\newtheorem{cor}{Corollary}
\newtheorem{assumption}{A.}

\begin{document}
\maketitle
\thispagestyle{empty}
\begin{abstract}
The paper studies the highly prototypical Fictitious Play (FP) algorithm, as well as a broad class of learning processes based on best-response dynamics, that we refer to as FP-type algorithms.
A well-known shortcoming of FP is that, while players may learn an equilibrium strategy in some abstract sense, there are no guarantees that the period-by-period strategies generated by the algorithm actually converge to equilibrium themselves. This issue is fundamentally related to the discontinuous nature of the best response correspondence and is inherited by many FP-type algorithms. Not only does it cause problems in the interpretation of such algorithms as a mechanism for economic and social learning, but it also greatly diminishes the practical value of these algorithms for use in distributed control.
We refer to forms of learning in which players learn equilibria in some abstract sense only (to be defined more precisely in the paper) as weak learning, and we refer to forms of learning where players' period-by-period strategies converge to equilibrium as strong learning. An approach is presented for modifying an FP-type algorithm that achieves weak learning in order to construct a variant that achieves strong learning. Theoretical convergence results are proved.
\end{abstract}

\section{Introduction}
Fictitious Play (FP), introduced in \cite{Brown51}, is one of the oldest and best-known game theoretic learning algorithms. FP has been shown to be an effective algorithm for distributed learning of Nash equilibria in various classes of games including two-player zero-sum games \cite{robinson1951iterative}, generic $2\times m$ games \cite{berger2005fictitious}, supermodular games
\cite{milgrom1990rationalizability,berger2008learning}, one-against-all games \cite{sela1999fictitious}, and potential games \cite{Mond96,benaim2005stochastic}. However, the manner in which players \emph{learn} in FP is often unsatisfactory, especially in the context of distributed control.

In FP, players learn equilibrium strategies in the sense that the time-averaged empirical distribution of players' actions converges to the set of Nash equilibria ---a form of learning known as \emph{convergence in empirical distribution}. This notion of learning tends to be problematic when the limit set of a learning algorithm contains mixed-strategy equilibria. In particular, convergence of the time-averaged empirical distribution to a mixed-strategy equilibrium does not imply any form of convergence in players' period-by-period strategies or actions. In practice, players' period-by-period strategies tend to move in progressively longer and longer cycles around an equilibrium set---the time-averaged empirical distribution is driven to equilibrium, but the period-by-period strategies never approach the equilibrium set themselves.

In the context of repeated-play algorithms, we refer to convergence of the empirical distribution (or some function thereof) to an equilibrium set as weak convergence, and we refer to any form of learning involving weak convergence as weak learning. We refer to the convergence of players' period-by-period strategies to an equilibrium set as strong convergence, and we refer to any form of learning involving strong convergence as strong learning. Intuitively speaking, weak learning means that players learn an equilibrium strategy in some abstract sense (i.e., convergence in empirical distribution) but may never actually implement the strategy they are learning. In strong learning, not only do players \emph{learn} an equilibrium strategy, but they also implement it.

FP is proven to achieve learning only in the weak sense, and thus no guarantees can be made regarding the convergence nor optimality of players period-by-period strategies. For example, Jordan \cite{jordan1993} presents a continuum of games for which FP achieves weak learning, yet in all but a countable subset of games, the period-by-period strategies produced by FP never approach the game's unique equilibrium. As another example, Young \cite{young2004strategic} presents a $2\times 2$ game in which FP achieves weak learning, but the period-by-period actions produced by FP achieve the lowest possible utility in every stage of the repeated play (see also Section \ref{sec_weak_convergence}).

Our first main contribution is the presentation of a simple variant of FP that converges strongly to equilibrium. In our strongly convergent variant of FP, players gradually and independently transition from using the FP best response rule to determine the next-iteration action, to using their current empirical distribution as a probability mass function from which they sample to determine the next-iteration action. We show that, for any game in which FP can be shown to converge weakly to equilibrium (and for which a certain robustness assumption holds---see \textbf{A.\ref{a_robustness}}), our variant of FP will converge strongly to equilibrium.

One advantage of this approach is that it is readily applicable to more general FP-type learning algorithms. Our second (and more general) main contribution is a method for taking a weakly convergent FP-type learning algorithm, and constructing from it, a strongly convergent variant.
We study a general class of FP-type algorithms and show that, so long as an algorithm achieves weak learning in a sufficiently robust sense (see \textbf{A.\ref{a_robustness}}), then a strongly convergent variant of the algorithm can be constructed. As an example of how the general result may be applied, we consider three weakly convergent FP-type algorithms---classical FP, Generalized Weakened FP \cite{leslie2006generalised}, and Empirical Centroid FP \cite{swenson2012ECFP,Swenson-MFP-Asilomar-2012}---and construct the strongly convergent variant of each.

\subsection{Related Work}
An overview of the topic of learning in games can be found in \cite{fudenberg1998theory,young2004strategic}. Various problems associated with learning mixed-strategy equilibria in best-response-type learning algorithms (including FP-type algorithms) are discussed in \cite{jordan1993}. In particular, the issue of weak convergence is considered, along with a discussion of some of the underlying mechanics that lead to weak convergence.

Many learning algorithms are designed to ensure that their limit points are pure-strategy equilibria \cite{marden06,marden-payoff,chasparis2010aspiration,pradelski2012learning,marden-shamma-08}. Ensuring convergence to a pure strategy is a natural way of ensuring strong learning, since weak learning can generally only occur when the limit set contains mixed strategies.

In contrast, this paper studies a method of ensuring strong convergence when the limit set of the algorithm contains mixed strategies. The ability to (strongly) learn mixed equilibria is important for many reasons, the foremost being that, in finite games, the set of Nash equilibria (NE) is only guaranteed to be non-empty if mixed equilibria are considered. Mixed strategies play an important role when the learned strategy needs to be robust to uncertainty in opponent behavior or game structure, or secure against the actions of malicious players \cite{rass2014numerical,voorneveld1999pareto,alpcan2010network,sela1999fictitious,dabcevic2014fictitious}. With regards to FP in particular, it was recently shown in \cite{candogan2013dynamics} that, for the class of near-potential games, the limit set of the FP dynamics (weakly speaking) is a neighborhood of a mixed equilibrium.

Regret-testing algorithms \cite{foster2003regret},\cite{germano2007global} achieve strong convergence to mixed-strategy equilibria in generic finite games. However, such algorithms operate on fundamentally different principles from FP-type algorithms---players implement a form of exhaustive search to coordinate on a NE strategy. Such algorithms tend to have slow convergence rates, especially when the number of players or available actions is large.

Stochastic FP (SFP)---introduced in \cite{Fud92}---was proposed as a learning mechanism that could (i) mitigate the problem of weak convergence to mixed equilibria in FP and (ii) provide a reasonable explanation for why real-world players might learn mixed-strategy equilibria. In SFP, the issue of weak convergence is addressed by smoothing each player's best response correspondence with the addition of small random shocks or perturbations. The stable points of SFP are not Nash equilibria, but rather Nash distributions. The set of Nash distributions converges to the set of Nash equilibria as the size of the perturbations goes to zero \cite{Fud92}. SFP has been shown to obtain strong convergence to the set of Nash distributions in various classes of games \cite{hofbauer2002global,benaim2005stochastic,fudenberg1998theory}. Moreover, if the perturbations are permitted to gradually decay throughout the course of the repeated play, then SFP converges to the set of NE \cite{leslie2006generalised}.

In contrast to SFP, the present work does not consider the descriptive agenda of providing an explanation for why real-world learners might act according to a given behavior rule. Furthermore, we present a simple and intuitive procedure for modifying a variety of weakly convergent learning algorithms in order to obtain a strong convergent variant. From a technical perspective, the current work differs from SFP in that the best response correspondence is not directly smoothed in any way.

The work \cite{leslie2006generalised} by Leslie  et al. studies a useful generalization of FP termed Generalized Weakened FP (GWFP). Among other contributions, the paper demonstrates that the convergence of FP is not affected by asymptotically decaying perturbations to players' best response sets. This result provides a cornerstone for our proofs by ensuring that FP (and GWFP) meet the critical robustness assumption \textbf{A.\ref{a_robustness}}. We study a strongly convergent variant of GWFP in Section \ref{sec_apps2}. Furthermore,  \cite{leslie2006generalised} also presents a payoff-based, actor-critic learning algorithm based on GWFP that achieves strong learning. Our work differs from this in that we provide a general method for constructing a strongly convergent algorithm from a weakly convergent one in a setting where instantaneous payoffs information may or may not be available.

Our preliminary results on strong convergence in FP is found in \cite{swenson2014strong}. The present work expands on \cite{swenson2014strong} by considering algorithms beyond classical FP and establishing more general conditions under which convergence can be attained (in particular, see \textbf{A.\ref{rho_a1}}--\textbf{A.\ref{rho_a3}}). Furthermore, \cite{swenson2014strong} contains a gap in reasoning in the proof of Lemma 2 which the present paper fills in.

The remainder of the paper is organized as follows. Section \ref{sec_prelims} sets up notation to be used in the subsequent development. Section \ref{sec_FP} introduces classical FP and discusses the problem of weak convergence in classical FP. Section \ref{sec_strong_fp} presents the strongly convergent variant of classical FP and states the strong convergence theorem for classical FP. Section \ref{sec_general_setup} presents the general notion of an FP-type algorithm, then presents the strongly convergent variant of an FP-type algorithm, states the general strong convergence result in the context of an FP-type algorithm, and presents the proof of the result. In Section \ref{sec_apps}, the general result is applied to prove strong convergence in classical FP, Generalized Weakened FP, and Empirical Centroid FP. Section \ref{sec_conclusion} concludes the paper.

\section{Preliminaries}
\label{sec_prelims}
\subsection{Setup and Notation}
A game in normal form is represented by the triple $\Gamma := (N,(Y_i,u_i)_{i\in N})$, where $N = \{1,\ldots,n\}$ denotes the set of players, $Y_i$ denotes the finite set of actions available to player $i$, and $u_i:\prod_{i\in N}Y_i \rightarrow \mathbb{R}$ denotes the utility function of player $i$. Denote by $Y:= \prod_{i\in N} Y_i$ the joint action space.

In order to guarantee the existence of Nash equilibria it is necessary to consider the mixed extension of $\Gamma$ in which players are permitted to play probabilistic strategies. Let $m_i := |Y_i|$ be the cardinality of the action space of player $i$, and let $\Delta_i := \{p\in \mathbb{R}^{m_i}:\sum_{k=1}^{m_i}p(k) = 1,~p(k)\geq 0 ~\forall k\}$ denote the set of mixed strategies available to player $i$---note that a mixed strategy is probability distribution over the action space of player $i$. Denote by $\Delta^n := \prod_{i\in N} \Delta_i$, the set of joint mixed strategies.

In this context, we often wish to retain the notion of playing a deterministic action. For this purpose, let $A_i := \{e_1,\ldots,e_{m_i}\}$ denote the set of ``pure strategies'' of player $i$, where $e_j$ is the $j$-th cannonical vector containing a $1$ at position $j$ and zeros otherwise.

The mixed utility function of player $i$ is given by $U_i(p) := \sum_{y \in Y} u_i(y) p_1(y)\ldots p_n(y)$, where $U_i:\Delta^n \rightarrow \mathbb{R}$. When convenient we sometimes write $U_i(p)$ as $U_i(p_i,p_{-i})$, where $p_i$ denotes the mixed strategy of player $i$ and $p_{-i}$ denotes the mixed strategies of all other players.
The set of Nash equilibria is given by $NE := \{p\in \Delta^n: U_i(p_i,p_{-i}) \geq U_i( p_i',p_{-i}), ~\forall p_i' \in \Delta_i,~\forall i\in N\}$.
Let
\vskip-15pt
\begin{equation}
BR_i^{\epsilon}(p_{-i}) := \{a_i \in A_i: U(a_i,p_{-i}) \geq \max_{\alpha_i \in A_i} U(\alpha_i,p_{-i})-\epsilon\}
\label{BR_epsilon_set}
\end{equation}
\vskip-5pt
\noindent
be the $i$-th players set of $\epsilon$-best responses to a strategy profile $p_{-i}$ adopted by the other players. Note that in this definition we only consider pure-strategy $\epsilon$-best responses.
Denote by $v_i(p_{-i}) := \max_{p_i\in \Delta_i} U_i(p_i,p_{-i}),$ the value obtained by playing a best response.

Throughout, we assume there exists a probability space $(\Omega,\mathcal{F},\mathbb{P})$ rich enough to carry out the construction of the various random variables required in this paper. For a random object $X$ defined on a measurable space $(\Omega,\mathcal{F})$, let $\sigma(X)$ denote the $\sigma$-algebra generated by $X$ \cite{williams_book}. As a matter of convention, all equalities and inequalities involving random objects are to be interpreted almost surely (a.s.) with respect to the underlying probability measure, unless otherwise stated.

\subsection{Repeated Play}
Suppose players repeatedly face off in the game $\Gamma$. Denote by $t\in \{1,2,\ldots\}$ a round of the repeated play. Let $\{a_i(t)\}_{t\geq 1}$ denote the sequence of actions taken by player $i$, where $a_i(t) \in A_i$, and let $\{a(t)\}_{t\geq 1}$, $a(t) = (a_1(t),\ldots,a_n(t))$ denote the sequence of joint actions.

Let $\{\mathcal{F}_t\}_{t\geq 1}$ be a filtration (sequence of $\sigma$-algebras) that contains the information available to players in round $t$ of the repeated play.
For $t\geq 1$ and $\alpha_i \in A_i$, let $g(\alpha_i,~t)\in\mathbb{R}$ be an $\mathcal{F}_{t-1}$-measurable random variable with $g_i(\alpha_i,~t) := \mathbb{P}(a_i(t) = \alpha_i\vert \mathcal{F}_{t-1})$, and let $g_i(t)\in \Delta_i$ be the vector with components $g_i(t) := (g_i(\alpha_1,~t),\ldots,g_i(\alpha_{m_i},~t))$, where $m_i$ is the cardinality of $A_i$.
We say $g_i(t)$ is the mixed strategy used by player $i$ in round $t$, and we say $\{g_i(t)\}_{t\geq}$ is the sequence of period-by-period (mixed) strategies used by player $i$. The sequence of joint period-by-period strategies is given by $\{g(t)\}_{t\geq 1}$, $g(t) := (g_1(t),\ldots,g_n(t))$.

Denote by $q_i(t) \in \Delta_i$, the empirical distribution of player $i$. The precise manner in which the empirical distribution\footnote{The term \emph{empirical distribution} is often used to refer explicitly to the time-averaged histogram of the action choices of some player $i$; i.e., $q_i(t) = \frac{1}{t}\sum_{s=1}^t a_i(s)$. Here, we allow for a broader definition that will permit interesting and useful algorithmic generalizations.} is formed will depend on the algorithm at hand. In general, $q_i(t)$ is formed as a function of the action history $\{a_i(s)\}_{s=1}^t$ and serves as a compact representation of the action history of player $i$ up to and including the round $t$. The joint empirical distribution is given by $q(t) := (q_1(t),\ldots,q_n(t))$.

Unless otherwise stated, $d(\cdot,~\cdot)$ denotes the standard Euclidean norm. For $m\geq 1$ and $S \subset \mathbb{R}^m$ define the distance from $p\in \mathbb{R}^m$ to $S\subset\mathbb{R}^m$ by $d(p,~S) := \inf\{d(p,~p'):~p'\in S\}$.  We say a repeated-play learning process converges \emph{weakly} to equilibrium if for some map $f:\Delta^n\rightarrow\Delta^n$ there holds $d(f(q(t)),~NE)\rightarrow 0$ as $t\rightarrow \infty$. In most cases in this paper, $f$ will simply be the identity function. We say a repeated-play learning process converges \emph{strongly}\footnote{The notion of strong convergence presented in this paper is comparable to the notions of ``convergence in intended behavior'' presented in \cite{Fud92} and ``convergence in strategic intentions'' given in \cite{young2004strategic}.} to equilibrium if $d(g(t),~NE)\rightarrow 0$ as $t\rightarrow \infty$. Note that weak learning implies that players \emph{learn} an equilibrium strategy, but may never actually begin to implement the strategy that is being learned. On the other hand, in strong learning players both \emph{learn} an equilibrium strategy, and implement the strategy that is being learned (see Section \ref{sec_weak_convergence} for more details).

\section{Fictitious Play}
\label{sec_FP}
\subsection{Fictitious Play}
\label{sec_FP_subsection}
Let
\vskip-20pt
\begin{equation}
\label{q_FP}
q_i(t) := \frac{1}{t}\sum_{s=1}^t a_i(s),
\end{equation}
\vskip-5pt
\noindent be the normalized histogram\footnote{Recall that the actions $a_i(t)\in A_i$ are dirac distributions in the mixed-strategy space $\Delta_i$.} of the actions of player $i$.

FP may be intuitively understood as follows. Players repeatedly face off in a stage game $\Gamma$. In any given stage of the game, players choose a next-stage action by assuming (perhaps incorrectly) that opponents are using stationary and independent strategies. Thus, in FP, players use the marginal empirical distribution of each opponent's past play, $q_i(t)$, as a prediction of the opponent's behavior in the upcoming round and choose a next-round strategy which is a best response against this prediction.

A sequence of actions $\{a(t)\}_{t\geq 1}$ such that\footnote{In all variants of FP discussed in this paper, the initial action $a_i(1)$ may be chosen arbitrarily for all $i$.\label{footnote_initial_cond}}
\vskip-15pt
\begin{equation}
a_i(t+1) \in BR_i(q_{-i}(t)),~\forall i,
\label{FP_BR}
\end{equation}
\vskip-5pt
\noindent for all $t\geq 1$, is referred to as a \emph{fictitious play process}. FP has been studied extensively to determine the classes of games for which it can be said to converge (weakly) to the set of Nash equilibria. 
Among other results, it has been shown that FP leads to weak learning in two-player zero-sum games \cite{robinson1951iterative}, potential games \cite{Mond96}, and generic $2\times m$ games \cite{berger2005fictitious}. We summarize these results in the following theorem.
\begin{theorem}
Let $\Gamma = (N,\{u_i(\cdot)\}_{i\in N},Y^n)$ be a two-player zero-sum game, potential game, or generic $2\times m$ game, and let $\{a(t)\}_{t\geq 1}$ be a fictitious play process on $\Gamma$. Then $d(q(t),~NE)\rightarrow 0$ as $t\rightarrow \infty$.
\label{theorem_classical_fp}
\end{theorem}

\subsection{Weak Convergence in Fictitious Play}
\label{sec_weak_convergence}
The following example (see \cite{young2004strategic}, p. 78), while fairly simple, clearly illustrates the phenomenon of weak convergence in FP, and demonstrates why weak convergence can be a deeply unsatisfactory notion of learning.


\begin{wrapfigure}{l}{0.35\textwidth}
  \begin{center}
    \includegraphics[width=0.28\textwidth]{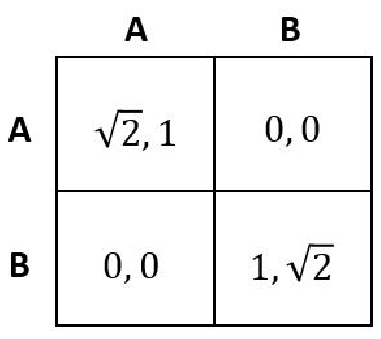}
  \end{center}
  \caption{}\label{fig:miscoord_game}
\end{wrapfigure}

Consider the two-player asymmetric coordination game shown in Figure \ref{fig:miscoord_game}. The game has three Nash equilibria: both players play A, both players play B, and an asymmetric mixed-strategy Nash equilibrium. The game is a potential game \cite{Mond96} (in fact, an identical interests game \cite{Mond01}) and hence falls within the purview of Theorem \ref{theorem_classical_fp}---regardless of the initial conditions, players engaged in an FP process will learn an equilibrium in the weak sense that $d(q(t),~NE)\rightarrow 0$ as $t\rightarrow\infty$.

Suppose that the players are engaged in an FP process on this game, and in the first round they miscoordinate their actions (e.g., one chooses A, and the other chooses B). Young \cite{young2004strategic} shows the somewhat counterintuitive result that the FP dynamics will in fact lead players to miscoordinate their action choices in every subsequent round of the learning process. Thus, despite the fact that $\lim_{t\rightarrow\infty} d(q(t),~NE) = 0$, the players' realized action choices are extremely suboptimal---yielding the lowest possible utility in each round of play. Intuitively speaking, this phenomenon occurs when players' actions cycle in such a way as to drive the time-averaged empirical distribution to a mixed-strategy Nash equilibrium, yet player's period-by-period strategies never constitute (nor even approach) a Nash equilibrium themselves.

It may be said that in weak learning players ``learn'' a NE strategy in some abstract sense, but never actually implement the strategy they are learning. In strong learning, players not only learn a NE strategy, but they also physically implement the strategy that is being learned.

The following section presents a simple modification of FP that achieves strong learning; i.e., players' period-by-period strategies converge to equilibrium in addition to convergence of the empirical distributions.


\section{Strong Convergence in Classical Fictitious Play}
\label{sec_strong_fp}

Consider a variant of FP in which the action for player $i$ at time $t$ is chosen by drawing a random sample from the mixed strategy (i.e., probability distribution) $g_i(t)$, where
\begin{equation}
g_i(t) \in BR_i(q_{-i}(t-1))\rho_i(t) + q_i(t-1)(1-\rho_i(t)),
\label{strong_FP_informal}
\end{equation}
$\rho_i(t) \in [0,1]$, and $\lim_{t\rightarrow\infty} \rho_i(t) = 0$. Intuitively, this is similar to the classical FP process \eqref{FP_BR}, but rather than playing a deliberate best response each round, players gradually transition toward drawing their stage $t$ action as a random sample from their own empirical distribution, $q_i(t)$.

The idea is that players will play a best response sufficiently often so that, per FP, the empirical distribution $q(t)$ will be driven toward equilibrium, as in Theorem \ref{theorem_classical_fp}. Then, since $\rho_i(t) \rightarrow 0$ as $t\rightarrow \infty$, the mixed strategy $g_i(t)$ tends towards $q_i(t)$, which is itself tending towards equilibrium. Informally, \eqref{strong_FP_informal} captures the main idea of strongly convergent FP. A formal presentation of the algorithm is given below.

\subsection{Strongly Convergent Variant of Classical FP}
\label{sec_strong_FP_construction}
Consider a variant of FP in which the action for player $i$ at time $t$ is chosen according to the following randomized rule:
\vskip-20pt
\begin{equation}
a_i(t) \sim g'_i(t) :=
\begin{cases}
b_i(t-1), & \mbox{ if } X_i(t) = 1,\\
q_i(t-1), & \mbox{otherwise},
\end{cases}
\label{action_rule0}
\end{equation}
where
$b_i(t-1) \in BR_i(q_{-i}(t-1)),$ the notation
$a_i(t) \sim g'_i(t)$ indicates that the action $a_i(t)$ is drawn as a random sample\footnote{The action $a_i(t) \in A_i$ is technically a dirac distribution over the finite action space $Y_i$ (see Section \ref{sec_prelims}), and the mixed strategy $g_i'(t)$ is a probability distribution over $Y_i$. More precisely, the notation $a_i(t) \sim g_i'(t)$ means that an action $y_i(t)$ is drawn as a random sample from $g_i'(t)$ with $a_i(t) := \delta_{y_i(t)}(y_i)$, where $\delta_{y_i(t)}(y_i) = 1$ if $y_i = y_i(t)$ and $\delta_{y_i(t)}(y_i) = 0$ otherwise.} from the probability mass function $g'_i(t)$, $X_i(t) \in \{0,1\}$ is a random variable, and $q_i(t)$ is the player's empirical distribution as defined in \eqref{qt_update2} below.
Let $\mathcal{F}_t := \sigma(\{a(s),X_1(s),\ldots,X_n(s),b_1(s),\ldots,b_n(s)\}_{s\leq t}),$
and note that $g'_{i}(t)$ is $\mathcal{F}_{t}$-measurable.
Let
\vskip-20pt
\begin{equation}
\rho_i(t) := \mathbb{P}(X_i(t) = 1\vert~\mathcal{F}_{t-1}),
\end{equation}
\vskip-5pt
\noindent and note that $\rho_i(t)$ is $\mathcal{F}_{t-1}$-measurable.
Intuitively speaking, $\rho_i(t)$ represents the probability that player $i$ deliberately chooses to play a best response strategy in round $t$ given the history of play up through the previous round.
We make the following assumptions regarding each player's probability of deliberately choosing a best response:
\begin{assumption}
$\lim\limits_{t\rightarrow\infty} \rho_i(t) = 0, ~\forall i\in N$, a.s.,
\label{rho_a1}
\end{assumption}
\begin{assumption}
$\sum\limits_{t\geq 1} \rho_i(t) = \infty, ~\forall i\in N$, a.s.,
\label{rho_a2}
\end{assumption}
\begin{assumption}
$\lim\limits_{t\rightarrow\infty} \frac{\sum_{k=1}^t\rho_i(k)}{\sum_{k=1}^t\rho_j(k)} = 1, ~\forall i,j \in N,$ a.s.
\label{rho_a3}
\end{assumption}

The first assumption ensures that players eventually transition towards playing their next-stage action as a sample from their empirical distribution rather than playing a deliberate best response. The second assumption ensures that, for each player, a deliberate best response is played infinitely often. The third assumption ensures that the number of deliberate best responses taken by each player remain relatively in sync.\footnote{Note that since $\rho_i(t)$ is only required to be $\mathcal{F}_{t-1}$-measurable, this parameter is in fact adaptively tunable. This is a feature of practical interest since it allows players to adjust their deliberate best response rates on the fly---possibly adapting to the (initially unknown) deliberate best response rates of others and to underlying process dynamics---in order to satisfy \textbf{A.\ref{rho_a1}}--\textbf{A.\ref{rho_a3}}.}
In practice, players may choose their deliberate best responses completely asynchronously; for example, setting $\rho_i(t) = 1/t^{r},~\forall i$, with $r\in (0,1]$, results in (purely) independent sampling of deliberate best response rounds and secures \textbf{A.\ref{rho_a1}}--\textbf{A.\ref{rho_a3}}.


Let
\vskip-25pt
\begin{equation}
\ell_i(t) := \sum\limits_{k=1}^{t} X_i(k)
\label{ell_def}
\end{equation}
\vskip-5pt
\noindent count the number of times player $i$ has deliberately played a best response until and including round $t$. Note that $\ell_i(t)$ is $\mathcal{F}_t$-measurable.  The empirical distribution $q_i(t)$ is defined recursively as\footnote{To initialize the process, let the action $a_i(1)$ be chosen arbitrarily, let $q_i(1) = a_i(1)$, and let $X_i(1) = 1$ for all $i$.}
\vskip-15pt
\begin{equation}
q_i(t+1) = q_i(t) + \frac{1}{\ell_i(t+1)}\left(a_i(t+1) - q_i(t)\right)X_i(t+1).
\label{qt_update2}
\end{equation}
\vskip-5pt
\noindent Intuitively speaking, the empirical distribution \eqref{qt_update2} is updated only over rounds when a deliberate best response was played. Note that $q_i(t)$ is $\mathcal{F}_t$-measurable.\footnote{Note that, \eqref{action_rule0} implicitly assumes that players have knowledge of the empirical distributions of opponents when computing a best response. This may be accomplished by assuming that players actions are accompanied with a ``tag'' indicating whether or not the played action was a deliberate best response. Alternatively, the information regarding $q_i(t)$ may tracked by the individual player $i$ and disseminated by a gossip-type algorithm \cite{swenson2012ECFP} or implicitly disseminated through a payoff-based scheme.}

Finally, let
\vskip-20pt
\begin{equation}
g_i(t) := b_i(t-1)\rho_i(t) + q_i(t-1)(1-\rho_i(t)),
\label{g_def}
\end{equation}
\vskip-5pt
\noindent and note that $g_i(t)$ is $\mathcal{F}_{t-1}$ measurable.\footnote{To see this, note first that $q_i(t-1)$ and $\rho_i(t)$ have been shown to be $\mathcal{F}_{t-1}$ measurable. Furthermore, this implies that $BR_i(q_i(t-1))$ is $\mathcal{F}_{t-1}$-measurable. Lastly, by construction, $b_i(t) \in BR_i(q_i(t-1))$ is $\mathcal{F}_{t-1}$-measurable.} More importantly, note that for every $\alpha_i \in A_i$,
$g_i(\alpha_i,t) = \mathbb{P}(a_i(t) = \alpha_i\vert~\mathcal{F}_{t-1}),$
and thus $g_i(t)$ represents the mixed strategy (conditioned on past play) used by player $i$ in round $t$. The joint mixed strategy used in round $t$ is given by $g(t) := (g_1(t),\ldots,g_n(t))$.

We refer to a process where, for each player $i$, $a_i(t)$ is updated according to \eqref{action_rule0}, $q_i(t)$ is updated according to \eqref{qt_update2}, and $g_i(t)$ is updated according to \eqref{g_def} as the strongly convergent variant of (classical) FP (for reasons to be clear soon).

\subsection{Strong Convergence in Classical FP: Main Result}
The following result states that in the strongly convergent variant of FP, players' period-by-period mixed strategies converge to the set of Nash equilibria---i.e., strong learning is achieved.
\begin{cor}
Let $\Gamma$ be a two-player zero-sum game, potential game, or generic $2\times m$ game. Assume \textbf{A.\ref{rho_a1}}--\textbf{A.\ref{rho_a3}} hold. Then the strongly convergent variant of FP achieves strong learning in the sense that $\lim_{t\rightarrow\infty} d(g(t),~NE) = 0$ almost surely.
\label{theorem_main_result}
\end{cor}

In order to prove the above result, we first study a more general notion of fictitious play and then prove the result as a corollary of the general theorem (see Theorem \ref{theorem_general_result}). Taking this general approach allows our strong convergence results to be be applied to other FP-type algorithms, e.g., Generalized Weakened FP (Section \ref{sec_apps2}) and Empirical Centroid FP (Section \ref{sec_apps3}). The proof of Corollary \ref{theorem_main_result} is given in Section \ref{sec_apps1}.

\subsection{Simulation Example}
In order to demonstrate the learning properties of strongly convergent FP, we simulated classical FP and strongly convergent FP in a simple two-player matching pennies game with utility functions as shown in Figure \ref{fig:payoff_matrix}. The game has a unique (symmetric) mixed-strategy equilibrium in which both players choose either action with probability $1/2$. Figure \ref{fig:FP_strategies} shows the period-by-period strategies generated by classical FP. Players' strategies are always pure and progress in continuously lengthening cycles. While the time-averaged empirical distribution is being driven to equilibrium, the period-by-period strategies clearly are not.

Figure \ref{fig:strong_FP_strategies} shows the period-by-period strategies generated by strongly convergent FP with $\rho(t) = t^{-.35}$. Players' period-by-period strategies are converging to the unique Nash equilibrium of the game.

Figure \ref{fig:received_utility} shows the utility received by the realized joint action $a(t)$ in each round of repeated play for both learning algorithms. The received payoffs in classical FP cycle around the value of the game, while the received payoffs in strongly convergent FP converge to the value of the game.

One possible tradeoff in strongly convergent FP is that less frequent deliberate best response actions and less frequent updating of the empirical distribution (see \eqref{qt_update2}) may lead to a slow-down in convergence rate. The empirical distribution processes for player $1$ in each algorithm is shown in Figure \ref{fig:empirical_dist} with $\rho(t) = t^{-.35}$.
{\small
\begin{figure}
        \centering
        \begin{subfigure}[b]{0.27\textwidth}
                \includegraphics[width=\textwidth]{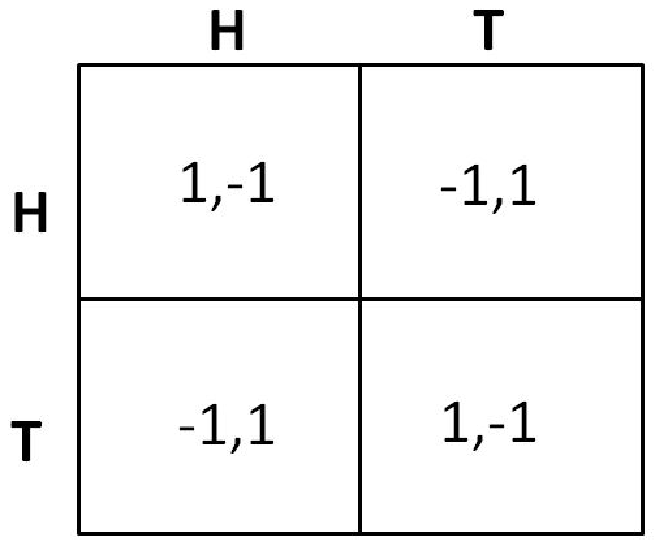}
                \caption{}
                \label{fig:payoff_matrix}
        \end{subfigure}%
        ~ 
        \begin{subfigure}[b]{0.3\textwidth}
                \includegraphics[width=\textwidth]{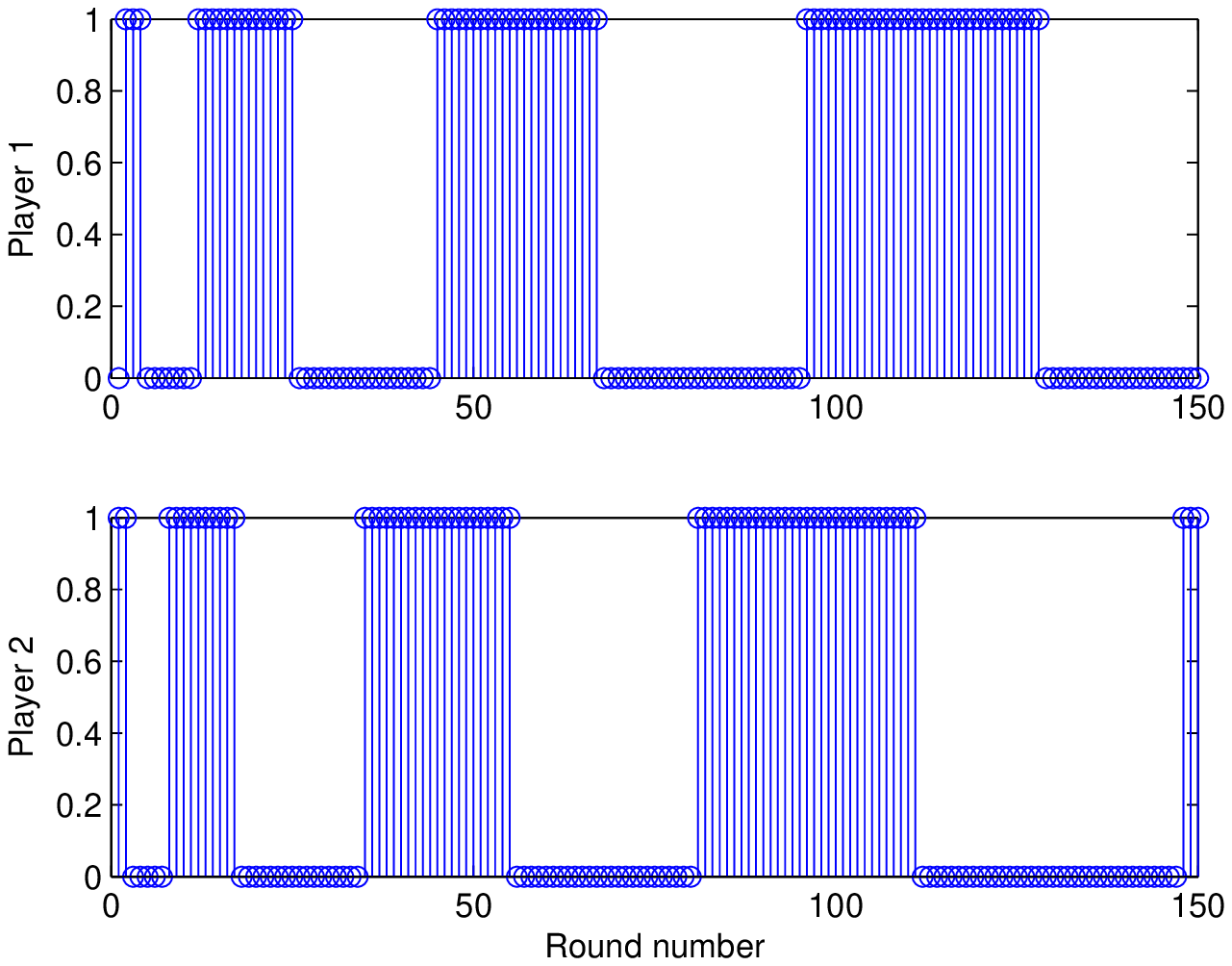}
                \caption{}
                \label{fig:FP_strategies}
        \end{subfigure}
        ~ 
        \begin{subfigure}[b]{0.3\textwidth}
                \includegraphics[width=\textwidth]{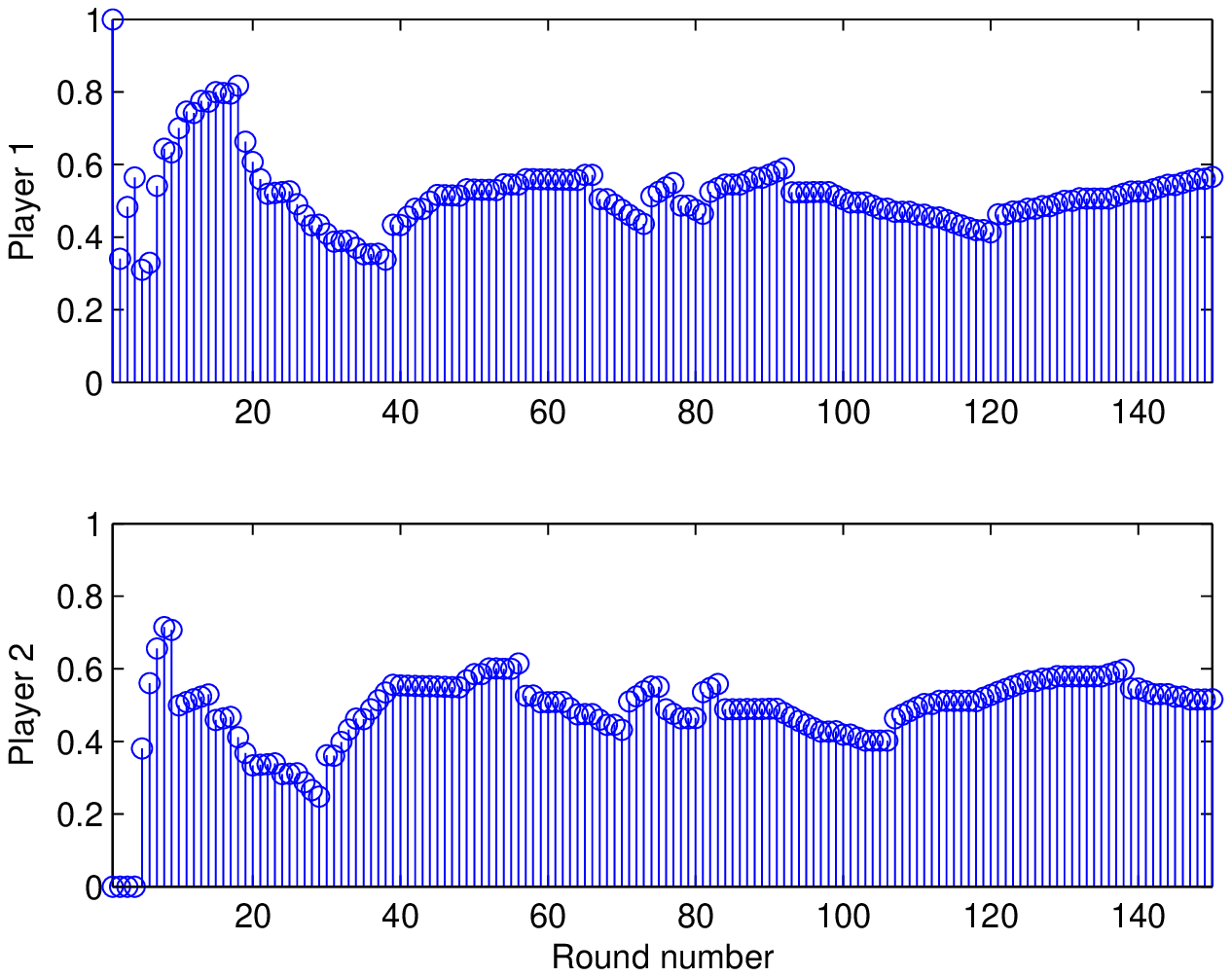}
                \caption{}
                \label{fig:strong_FP_strategies}
        \end{subfigure}
        \begin{subfigure}[b]{0.3\textwidth}
                \includegraphics[width=\textwidth]{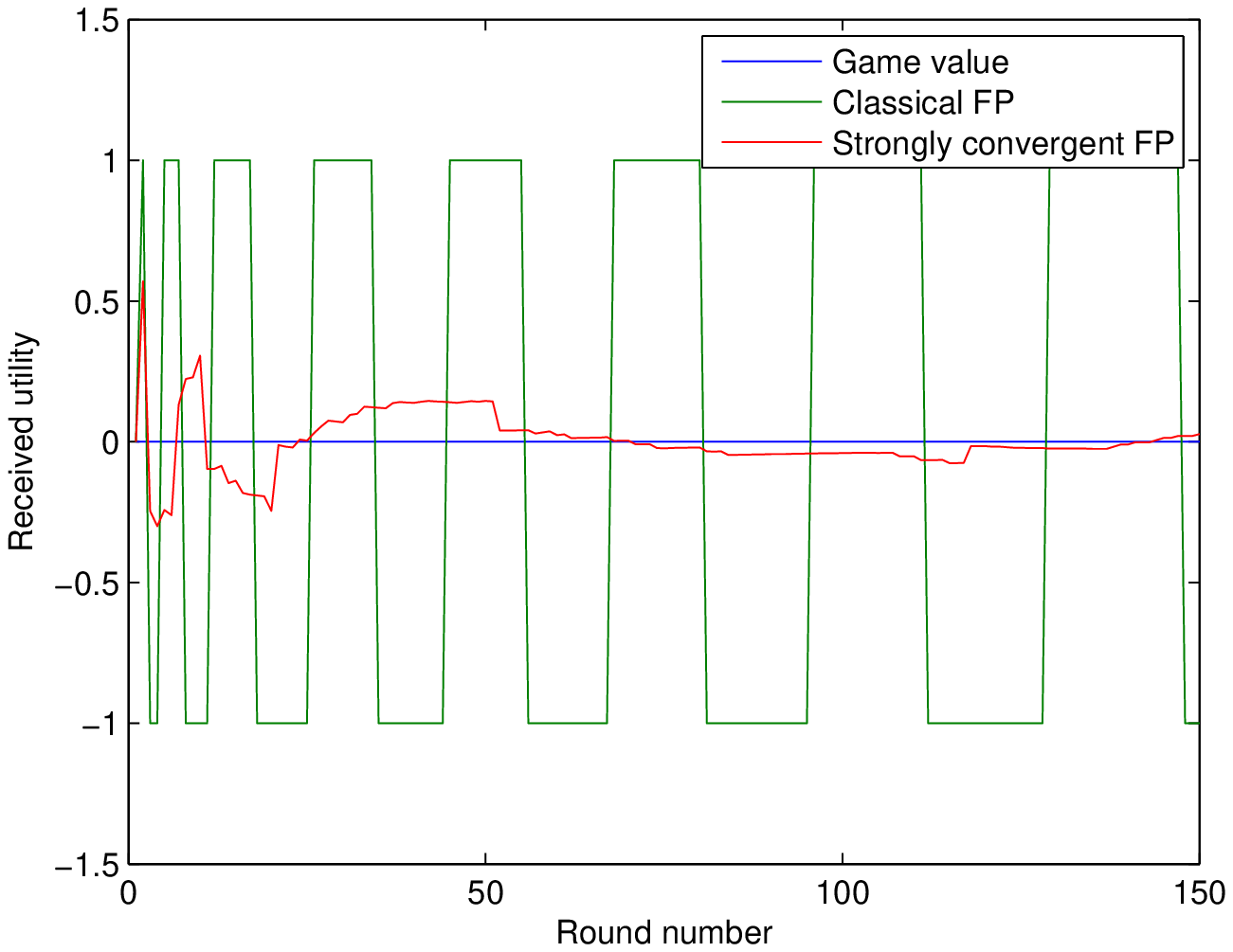}
                \caption{}
                \label{fig:received_utility}
        \end{subfigure}
        \begin{subfigure}[b]{0.3\textwidth}
                \includegraphics[width=\textwidth]{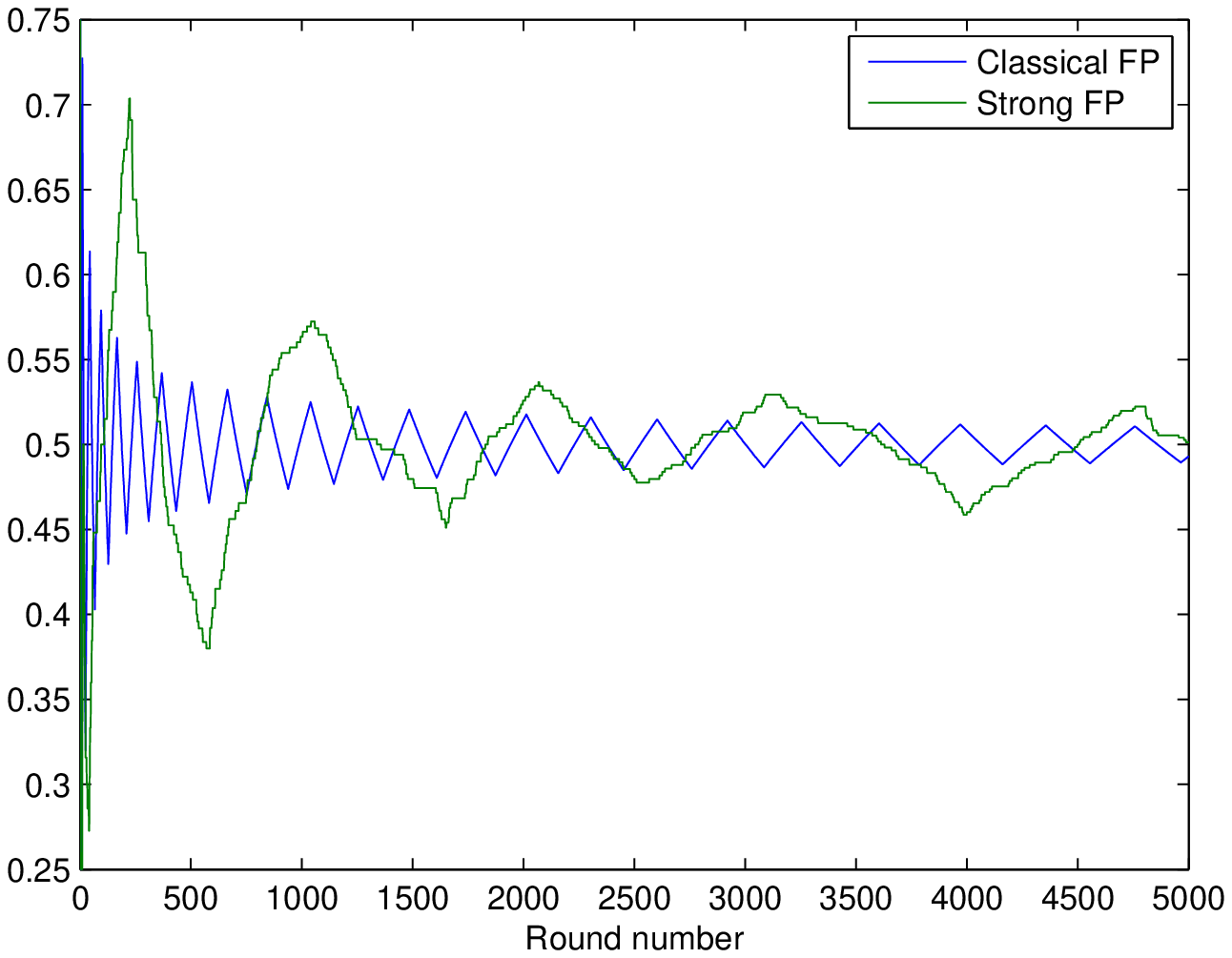}
                \caption{}
                \label{fig:empirical_dist}
        \end{subfigure}
        \caption{\small \ref{fig:payoff_matrix}: Matching pennies payoff matrix, \ref{fig:FP_strategies}: The probability of each player playing heads in round $t$ using the classical FP algorithm, \ref{fig:strong_FP_strategies}: The probability of each player playing heads in round $t$ using the strongly convergent FP algorithm, \ref{fig:received_utility}: The received utility in round $t$ given the realized action $a(t)$, \ref{fig:empirical_dist}: The empirical distribution process of the action H (heads) for player $1$ in both FP and strongly convergent FP.}\label{fig:simulation}
\end{figure}
}
\section{General Setup}
\label{sec_general_setup}
In this section we study strong learning in FP-type algorithms ---a class of algorithms that generalizes FP and includes many learning processes based on best-response dynamics.\footnote{The class of FP-type algorithms proposed here is similar in spirit to the class of best-response-based algorithms considered in \cite{jordan1993}.} In Section \ref{sec_FP_type}, we define the notion of an FP-type algorithm. In Section \ref{sec_FP_type_examples} we present some examples of an FP-type algorithm. In Section \ref{sec_FP_type_strong} we define the strongly convergent variant of an FP-type algorithm. In Section \ref{sec_FP_type_main_result} we provide the general strong convergence result for an FP-type algorithm (see Theorem \ref{theorem_general_result}), and in Sections \ref{sec_additional_defs}--\ref{sec_proof_general_result} we prove the general result.

\subsection{FP-Type Algorithm}
\label{sec_FP_type}
An FP-type algorithm generalizes classical FP in the following ways: (i) the notion of a player's empirical distribution is generalized, (ii) players are permitted to use a function of the empirical distribution (rather than use the empirical distribution itself) as a predictor of the next-round strategy of opponents, (iii) convergence to equilibrium may occur in terms of a function of the empirical distribution (rather than convergence to equilibrium of the empirical distribution itself), and (iv) limit sets other than the set of NE are permitted.

We define an FP-type algorithm as follows. Let players be engaged in repeated play of a stage game $\Gamma$. Let $a_i(t)$ represent the action of player $i$ in round $t\in\{1,~2,\ldots\}$, and let
$H_i(t) := \{a_i(s)\}_{s=1}^t$
represent the action history of player $i$ up to and including round $t$.

In classical FP, for each player $i$, the normalized histogram of the player's action choices \eqref{q_FP} is used as a compact representation of the player's action history. In the general formulation of an FP-type algorithm, we still suppose that players track a compact representation of the action history, but we allow the compact representation to take on a fairly general form,\footnote{In most literature, the notion of an \emph{empirical distribution} refers strictly to the time-averaged empirical histogram of a player's action choices, as in classical FP \eqref{q_FP}. However, as discussed in Section \ref{sec_prelims}, we use the term empirical distribution more generally to refer to an arbitrarily formed (see \textbf{A.\ref{a_general_q}}) distribution that a player uses to track information regarding opponents' empirical action histories. This abuse of terminology allows us to more naturally extend concepts to the general FP-type setting.}  as stated in the following assumption:
\begin{assumption}
\label{a_general_q}
The empirical distribution of player $i$ is of the form $q_i(t) := f^q_i(H_i(t),~t)$, where $f^q_i(\cdot,~t):\prod_{s=1}^t A_i \rightarrow \Delta_{i}$.
\end{assumption}
We make the following assumption regarding the sequence of functions $\{f^q_i(\cdot,~t)\}_{t\geq 1}$ used to form the empirical distribution sequence of player $i$:
\begin{assumption}
\label{a_step_size_bound}
For any history sequence $\{H_i(t)\}_{t\geq 1}$ for player $i$, there holds $\lim_{t\rightarrow\infty} \|f^q_i(H_i(t+1),~t+1) - f^q_i(H_i(t),~t)\| = 0$.
\end{assumption}

In particular, this implies that---regardless of the action history---there holds\\ $\lim_{t\rightarrow \infty} \|q_i(t+1) - q_i(t)\| = 0$ for each player $i$. This fairly mild assumption captures the essential characteristics required for our asymptotic analysis, and may be seen as a generalization of classical FP where exact averaging of actions over time yields $\|q_i(t+1) - q_i(t)\| \leq \frac{1}{t}$ (see Section \ref{sec_FP_example}). Together, assumptions \textbf{A.\ref{a_general_q}}--\textbf{A.\ref{a_step_size_bound}} allow us to consider a variety of FP inspired algorithms, including those with general step sizes \cite{leslie2006generalised} and those with more intricate history dependent rules such as derivative action \cite{Arslan04}.

In an FP-type algorithm, players form a prediction of the future behavior of opponents as a function of the current empirical distribution. Let $p_i(t)$ be player $i$'s prediction of opponent strategies for the upcoming round $(t+1)$. We assume,
\begin{assumption}
Player $i$'s prediction $p_i(t)$ of opponent behavior is of the form $p_i(t) = f_i^p(q(t))$, where $f_i^p:\Delta^n \rightarrow \Delta_{-i}$ is a Lipschitz continuous, time-invariant function.
\label{a_prediction}
\end{assumption}

We say a sequence of actions $\{a(t)\}_{t\geq 1}$ is an FP-type process if for all $i\in N$ and all $t\geq 1$,
$a_i(t+1) \in BR_i^{\epsilon_t}(p_i(t)),$
where $BR^{\varepsilon_t}_i(\cdot)$ is the $\varepsilon_t$-best response set (recall \eqref{BR_epsilon_set}), and $\{\epsilon_t\}_{t\geq 1}$ is a sequence satisfying $\lim_{t\rightarrow\infty} \epsilon_t = 0$.

In many variants of FP, including classical FP, learning occurs in the sense that $d(q(t),~NE)\rightarrow 0$. We generalize this notion of learning by allowing for limit sets other than the set of NE and allowing for convergence in terms of a function of $q(t)$ rather than permitting convergence only in terms of $q(t)$ itself.

Let $E$ be some target equilibrium set (not necessarily the set of NE). An FP-type process is said to learn elements of $E$ if for each $i$ there exists a function $f^\xi_i$ satisfying:
\begin{assumption}
The function $f^\xi_i:\Delta^n\rightarrow\Delta_i$ is Lipschitz continuous and time invariant,
\label{a_xi}
\end{assumption}
and such that, for $\xi_i(t) := f^\xi_i(q(t))$ and $\xi(t) := (\xi_1(t),\ldots,\xi_n(t))$ there holds
$\lim_{t\rightarrow 0} d(\xi(t),~E) =0.$
We refer to $\xi(t)$ as the asymptotic learning distribution, and $f^\xi_i$ as the convergence map of player $i$.

In general, we will denote an instance of an FP-type learning algorithm by $\Psi = (\{f^q_i(\cdot,~t)\}_{t\geq 1},f_i^p,f^\xi_i)_{i\in N}$.
In order to construct a strongly convergent variant of $\Psi$ we will require that $\Psi$ obtain weak convergence in sufficiently robust sense as stated in the following assumption.

\begin{assumption}
For the stage game $\Gamma$ and equilibrium set $E$, the FP-type algorithm $\Psi$ is such that for any sequence $(\epsilon_t)_{t\geq 1}$ satisfying $\lim_{t\rightarrow\infty} \epsilon_t = 0$, the FP-type algorithm $\Psi$ obtains weak convergence in the sense that $\lim_{t\rightarrow 0} d(\xi(t),~E)= 0$.
\label{a_robustness}
\end{assumption}

The above assumption ensures that the FP-type algorithm is robust to asymptotically decaying perturbations in a player's best response set. When studying the strongly convergent variant of $\Psi$ in the following section, the assumption \textbf{A.\ref{a_robustness}} will serve to ensure that convergence of the process is not disrupted by minor asynchronies in the number of deliberate best responses taken by each player (i.e., minor disparities in \eqref{ell_def}).

\subsection{Examples}
\label{sec_FP_type_examples}
\subsubsection{Classical Fictitious Play}
\label{sec_FP_example}
Classical FP (Section \ref{sec_FP_subsection}) fits the template of an FP-type algorithm with $q_i(t) = \frac{1}{t}\sum_{s=1}^t a_i(s)$. Note that $q_i(t)$ may be written in recursive form as: $q_i(t+1) = q_i(t) + 1/(t+1)\left(a_i(t+1) - q_i(t) \right)$. Thus, $\|q_i(t+1) - q_i(t)\| \leq \frac{M_i}{t+1}$, where $M_i := \sup_{p_i',p_i''\in \Delta_i} \|p_i' - p_i''\|$, and \textbf{A.\ref{a_step_size_bound}} is satisfied. The prediction map $f_i^p$ is given by the identity function, and convergence map $f^\xi_i$ also given by the identity function. The target equilibrium set is given by $E := NE$, the set of Nash equilibria.

\subsubsection{Generalized Weakened Fictitious Play}
\label{sec_GWFP_intro}
Leslie et al. \cite{leslie2006generalised} study a useful generalization of FP, termed Generalized Weakened FP (GWFP), in which players are permitted to choose a suboptimal best response each round, so long as the degree of suboptimality decays asymptotically to zero, and in which step-size sequences other than $\{1/t\}_{t\geq 1}$ are permitted.

Formally, for $p_{-i} \in \Delta_{-i}$ and $\epsilon\geq 0$, let\footnote{The set $\bar{BR^{\epsilon}_i}(p_{-i})$ defined below differs from the set $BR_i^\epsilon(p_{-i})$ defined in the preliminaries in that $\bar{BR^{\epsilon}_i}(p_{-i})$ includes all \emph{mixed} strategy best responses, whereas $BR_i^\epsilon(p_{-i})$ contains only the pure strategy best responses. The set $\bar{BR^{\epsilon}_i}(p_{-i})$ is used here in order to precisely define a GWFP process as given in \cite{leslie2006generalised}, but the remainder of the paper focuses on the set $BR_i^\epsilon(p_{-i})$.}
$\bar{BR^{\epsilon}_i}(p_{-i}) := \{p_i\in \Delta_i: U_i(p_i,p_{-i}) \geq \max_{\alpha_i \in A_i}U_i(\alpha_i,p_{-i})-\epsilon\},$
and for $p\in \Delta^n$, let $\bar{BR^{\epsilon}}(p) := (\bar{BR^{\epsilon}_1}(p_{-1}),\ldots,\bar{BR^{\epsilon}_n}(p_{-n}))$. A sequence $\{q(t)\}_{t\geq 1}$ is said to be a GWFP process if
$q(t+1) \in (1-\gamma(t+1))q(t) + \gamma(t+1)(\BRtildet(q(t)) + M_{t+1})$
with $\gamma(t) \rightarrow 0$ and $\epsilon_t \rightarrow 0$ as $t\rightarrow \infty$, $\sum_{t\geq 1} \gamma(t) = \infty$, and $\{M_t\}_{t\geq 1}$ is a deterministic (or stochastic) perturbation sequence satisfying
$\lim\limits_{t\rightarrow\infty}\sup_{k}\{\|\sum_{i=t}^{k-1} \gamma_{i+1} M_{i+1} \|:\sum_{i=t}^{k-1}\gamma_{i+1} \leq T\} = 0$ (a.s.).

We consider a special case of GWFP in which $M_t = 0, ~\forall t$ and the $\epsilon$-best response set is restricted to the set of pure strategy $\epsilon$-best responses. That is, we consider the subset of GWFP process such that
$a(t+1) \in BR^{\epsilon_t}(q_{-i}(t)),$
and,
\begin{equation}
\label{q_gwfp_def}
q(t+1) = q(t) + \gamma(t+1)\left(a(t+1) - q(t) \right),
\end{equation}
with $\epsilon_t \rightarrow 0$,
and in a slight variation of terminology we refer to the sequence of actions $\{a(t)\}_{t\geq 1}$ satisfying the above as a GWFP process.

In the terminology of Section \ref{sec_FP_type}, GWFP fits the template of an FP-type algorithm with the empirical distribution $q_i(t)$ defined recursively as in \eqref{q_gwfp_def} (where it is assumed that $\lim_{t\rightarrow\infty} \gamma(t) = 0$), the prediction map $f_i^p$ given by the identity function for all $i$, and the convergence map $f^{\xi}_i$ given by the identity function for all $i$, and the target equilibrium set is given by $E := NE$---the set of Nash equilibria.

\subsubsection{Empirical Centroid Fictitious Play---Learning Consensus Equilibria}
\label{sec_ECFP_intro}
Empirical Centroid FP (ECFP) was conceived as a variant of FP suited to implementation in large-scale games \cite{swenson2012ECFP,Swenson-MFP-Asilomar-2012}. In ECFP, rather than tracking the empirical distribution of each individual opponent (as in FP), players track and respond to only the centroid of the empirical distributions. In order to ensure the process is well defined the following assumption is made:
\begin{assumption}
\label{a_ident_strat}
All players use the same strategy space.
\end{assumption}
Under this assumption, let the empirical distribution be defined by
\vskip-20pt
\begin{equation}
\label{q_ecfp_def}
q_i(t) := \frac{1}{t}\sum_{s=1}^t a_i(s),
\end{equation}
\vskip-5pt
\noindent and let the empirical centroid distribution be defined by
$\bar q(t) := \frac{1}{n}\sum_{i\in N} q_i(t).$
We say a sequence of actions $\{a(t)\}_{t\geq 1}$ is an ECFP process if for all $i$ and all $t\geq 1$,
\begin{equation}
a_i(t+1) \in BR_i^{\epsilon_t}(\bar q_{-i}(t)),
\label{ecfp_process}
\end{equation}
where $\bar q_{-i}(t) = (\bar q(t),\ldots,\bar q(t)) \in \prod_{j\not= i} \Delta_j$ is the $(n-1)$-tuple containing $(n-1)$ repeated copies of $\bar q(t)$, and $\{\epsilon_t\}_{t\geq 1}$ is a sequence satisfying $\lim_{t\rightarrow\infty} \epsilon_t = 0$.

In ECFP, players learn elements of the set of consensus Nash equilibria\footnote{We assume here that the set of consensus Nash equilibria is non-empty. When revisiting ECFP in Section \ref{sec_apps3}, we provide an assumption on the utility structure that ensures that the set is indeed non-empty.}, defined by
$C:= \{p = (p_1,~\ldots~,p_n)\in NE:~ p_1 = p_2 = \ldots = p_n\},$
the subset of Nash equilibria in which all players use identical strategies (see \cite{swenson2012ECFP} for more details).
Define $\bar q^n(t) := (\bar q(t),\ldots,\bar q(t)) \in \Delta^n$ to be the $n$-tuple containing repeated copies of $\bar q(t)$; learning in ECFP takes place in the sense that $\lim_{t\rightarrow\infty} d(\bar q^n(t),~C) = 0$.

In the terminology of Section \ref{sec_FP_type}, ECFP fits the template of an FP-type algorithm with the empirical distribution given by \eqref{q_ecfp_def}, the prediction map $f_i^p$ given by
$f_i^p(q(t))  := \left(\frac{1}{n}\sum_{j\in N}q_j(t), \ldots, \frac{1}{n}\sum_{j\in N}q_j(t)\right),~ \forall i,$
where the right-hand side is a $(n-1)$-tuple containing repeated copies of $\bar q(t)$,
and the convergence map given by
$f^\xi_i(q(t)) := \frac{1}{n}\sum_{j=1}^n q_j(t),~\forall i.$
The target equilibrium set is given by $E:= C$, the set of consensus Nash equilibria.

\subsubsection{Empirical Centroid Fictitious Play---Learning Mean-Centric Equilibria}
\label{sec_ecfp_mce_intro}
In this section we consider a slight modification of the ECFP algorithm presented in Section \ref{sec_ECFP_intro} that enables players to learn elements of an alternate (non-Nash) equilibrium set.

Let an ECFP action process be defined as in \eqref{ecfp_process}. Define the set of mean-centric equilibria by
$MCE := \{p \in \Delta^n: ~U_i(p_i,~\bar p_{-i}) \geq U_i(p_i',~\bar p_{-i})~\forall p_i' \in \Delta_i,~\forall i\}.$
The set of MCE is neither a superset nor a subset of the NE---rather, it is a set of natural equilibrium points tailored to the ECFP dynamics \cite{swenson2013MCE}. The set of consensus Nash equilibria $C$ (see Section \ref{sec_ECFP_intro}) however, is contained in the set of MCE.

In ECFP, players learn elements of MCE in the sense that $\lim_{t\rightarrow\infty}d(q(t),~MCE)= 0$. In the terminology of Section \ref{sec_FP_type}, this fits the template of an FP-type algorithm with $q_i(t)$ given by \eqref{q_ecfp_def}, $f_i^p$ defined in the same way as in Section \ref{sec_ECFP_intro}, the convergence map $f^\xi_i$ given by the identity for all $i$, and the target equilibrium set given by $E := MCE$.

Note that the only difference between the ECFP algorithm discussed in the Section \ref{sec_ECFP_intro} and the ECFP algorithm discussed here is the choice of target equilibrium set $E$ and convergence maps $f^\xi_i$.

\subsection{Strongly Convergent Variant of an FP-type Algorithm}
\label{sec_FP_type_strong}
In this section we construct the strongly convergent variant of an FP-type learning algorithm. The construction here is a generalization of that of Section \ref{sec_strong_FP_construction} where we constructed the strongly convergent variant of classical FP.

Let $\Psi = (\{f^q_i(\cdot,~t)\}_{t\geq 1},~f_i^p,~f^\xi_i)_{i\in N}$ be an FP-type learning algorithm.
For each $i\in N$, let $\{X_i(t)\}_{t\geq 1}$ be a sequence of random variables with $X_i(t) \in \{0,1\}$. Analogous to Section \ref{sec_strong_fp}, $X_i(t)=1$ will serve to indicate that player $i$ took a deliberate best response in round $t$. Let
\vskip-20pt
\begin{equation}
\label{ell_def2}
\ell_i(t) := \sum_{s=1}^t X_i(s)
\end{equation}
\vskip-5pt
\noindent count the number of deliberate best responses taken by player $i$ through $t$.

In Section \ref{sec_strong_FP_construction} the empirical distribution of player $i$, \eqref{qt_update2}, is a time average taken only over rounds when player $i$ took a deliberate best response. In order to generalize this notion to an FP-type algorithm, define the term
\begin{equation}
\label{tau_def}
\tau_i(s) := \inf\{t:\ell_i(t)=s\}.
\end{equation}
For $s\geq 1$, $\tau_i(s)$ indicates the round when player $i$ took their $s$-th deliberate best response,\footnote{Note that by \eqref{tau_exist}, $\tau_i(s)$ is finite valued a.s. for any $s\in \{1,~2,\ldots\}$.}
and the sequence $\{\tau_i(s)\}_{s\geq 1}$ gives the subsequence of rounds when player $i$ took a deliberate best response. For $t\in \{1,~2,\ldots\}$ let
$\bar H_i(t) := \{a_i(\tau_i(s)):~\tau_i(s) \leq t\}$
denote the action history of player $i$. Note that $\bar H(t)$ records only the history of actions that were taken as deliberate best responses.
Let the empirical distribution of player $i$ at time $t$ be formed as
\vskip-18pt
\begin{equation}
\label{qt_general_strong_update}
q_i(t) := f_i^q(\bar H_i(t),~\ell_i(t)).
\end{equation}
\vskip-5pt
\noindent Let the asymptotic learning distribution (see \textbf{A.\ref{a_xi}} and subsequent discussion) be given by $\xi_i(t) := f_i^\xi(q(t))$ and  $\xi(t) := (\xi_1(t),\ldots,~\xi_i(t))$.

Let the action for player $i$ in round $t\geq 2$ be chosen according to the random rule\footnote{To initialize the process, let the action $a_i(1)$ be chosen arbitrarily, let $X_i(1) = 1$, and let $\bar H(1) = a_i(1)$ for all $i$.}
\begin{equation}
a_i(t) \sim g'_i(t) :=
\begin{cases}
b_i(t-1), & \mbox{ if } X_i(t) = 1,\\
\xi_i(t-1), & \mbox{otherwise},
\end{cases}
\label{action_rule1}
\end{equation}
where $p_i(t-1) = f_i^p(q(t-1))$, and $b_i(t-1) \in BR^{\eta_t}_i(p_i(t-1)),$
and assume:\footnote{Note that this assumption subsumes the more typical assumption that $\eta_t = 0,~\forall t$. By making this more general assumption we are able to handle interesting scenarios that may arise in a practical implementation of the algorithm; e.g., players have some asymptotically decaying error in their knowledge of their utility function or knowledge of opponent's empirical distributions.}
\begin{assumption}
\label{a_eta}
The sequence $(\eta_t)_{t\geq 1}$ associated with $b_i(t)$ of \eqref{action_rule1} is such that $\lim\limits_{t\rightarrow\infty} \eta_t = 0$.
\end{assumption}
Let $\mathcal{F}_t := \sigma(\{a(s),X_i(s),\ldots,X_n(s),b_1(s),\ldots,b_n(s)\}_{s\leq t}).$
Let the probability that player $i$ chooses a deliberate best response in round $t$ conditioned on past events be given by
$\rho_i(t) := \mathbb{P}(X_i(t) = 1\vert \mathcal{F}_{t-1}),$
and assume \textbf{A.\ref{rho_a1}}--\textbf{A.\ref{rho_a3}} hold.
Note that $q_i(t)$, $p_i(t)$, $\xi_i(t)$, and $g_i'(t)$ are $\mathcal{F}_t$--measurable and that by definition, $\rho_i(t)$ is $\mathcal{F}_{t-1}$--measurable.

Finally, let
\vskip-15pt
\begin{equation}
\label{g_def_general}
g_i(t) := b_i(t-1)\rho_i(t) + \xi_i(t)(1-\rho_i(t)).
\end{equation}
\vskip-5pt
\noindent Note that $g_i(t)$ is $\mathcal{F}_{t-1}$--measurable and that $g(\alpha_i,t) = \mathbb{P}(a_i(t) = \alpha_i\vert \mathcal{F}_{t-1})$; that is, $g_i(t)$ represents the mixed strategy in use by player $i$ in round $t$ (compare with \eqref{g_def}). Let $g(t) := (g_1(t),\ldots,g_n(t))$ denote the joint mixed strategy in use at time $t$.

We refer to a process where, for each player $i$, $q_i(t)$ is updated according to \eqref{qt_general_strong_update}, $a_i(t)$ is updated according to \eqref{action_rule1}, and $g_i(t)$ is updated according to \eqref{g_def_general} as the strongly convergent variant of $\Psi$ (for reasons to be clear soon---see Theorem \ref{theorem_general_result}). In Section \ref{sec_apps} we will demonstrate applications of this in the context of the previous examples.

\subsection{General Result}
\label{sec_FP_type_main_result}
The following theorem provides the general result from which the strong convergence of various FP-type algorithms can be derived.
\begin{theorem}
Let $\Gamma$ be a finite normal form game, let $E$ be an equilibrium set, and let $\Psi$ be an FP-type algorithm satisfying \textbf{A.\ref{a_general_q}}--\textbf{A.\ref{a_robustness}}. If the strongly convergent variant of $\Psi$ satisfies \textbf{A.\ref{rho_a1}}--\textbf{A.\ref{rho_a3}} and \textbf{A.\ref{a_eta}} then it achieves strong learning in the sense that $\lim_{t\rightarrow\infty} d(g(t),~E) = 0$, almost surely.
\label{theorem_general_result}
\end{theorem}

We emphasize that in the above result players' period-by-period mixed strategies $g(t)$ are converging to equilibrium. In general,
when seeking to construct the strongly convergent variant of some FP-type algorithm $\Psi$, the most challenging aspect of applying Theorem \ref{theorem_general_result} is the verification that $\Psi$ satisfies \textbf{A.\ref{a_robustness}}. The remaining assumptions \textbf{A.\ref{a_general_q}}--\textbf{A.\ref{a_xi}} are generally fairly trivial to verify. Assumptions \textbf{A.\ref{rho_a1}}--\textbf{A.\ref{rho_a3}} and \textbf{A.\ref{a_eta}} pertain to the manner in which the strongly convergent variant of $\Psi$ is constructed and are not related to intrinsic properties of $\Psi$ itself.

\subsection{Some Additional Definitions}
\label{sec_additional_defs}
In order to prove Theorem \ref{theorem_general_result} we will study the behavior of an underlying FP-type process that is embedded in the action, history, and empirical distribution processes produced by the strongly convergent variant of $\Psi$. In particular, for $i\in N$ and $s\in \{1,2,\ldots\}$, let $\tau_i(s)$ be defined as in \eqref{tau_def}, and define the following terms:
$\tilde a_i(s) := a_i(\tau_i(s)),~
 \tilde a(s) := (\tilde a_1(s),\ldots,\tilde a_n(s)),~
 \tilde H_i(s) := \bar H_i(\tau_i(s)),~
 \tilde q_i(s) := q_i(\tau_i(s)),~
 \tilde q(s) := (\tilde q_1(s),\ldots,\tilde q_n(s)),~
 \tilde p_i(s) := f^p_i(\tilde q(s)),~
 \tilde \xi(s) := (f^\xi_1(\tilde q(s)),\ldots,f^\xi_n(\tilde q(s))).$
The aforementioned terms (marked with a tilde) correspond to to the embedded FP-type process that we will study in the proof of Theorem \ref{theorem_general_result}. In particular, for each player $i$, the sequence $\{\tau_i(s)\}_{s\geq 1}$ denotes the subsequence of rounds when the player chose to play a deliberate best response. The sequence ${\tilde a_i(s)}_{s\geq 1}$ is the action sequence occurring along the subsequence of rounds when player $i$ chose to play a deliberate best response. The sequence $\{\tilde H_i(s)\}_{s\geq 1}$ corresponds to the action history of player $i$ along the same subsequence. The sequence $\{\tilde q_i(s)\}_{s\geq 1}$ corresponds to the empirical distribution of player $i$ along the same subsequence; in particular, note that by Lemma \ref{q_tilde_lemma} (see appendix), $\{\tilde q_i(s)\}_{s\geq 1}$ fits the format prescribed by \textbf{A.\ref{a_general_q}} for the embedded FP-type process: $\tilde q_i(s) = f_i^q(\tilde H(s),s).$
Finally, the term $\tilde \xi(s)$ is the asymptotic learning distribution associated with the embedded FP-type process.

In studying the embedded FP-type process, it will be important to characterize the terms to which players are best responding. With this in mind, note that per \eqref{action_rule1}, the action at time $\tau_i(s+1)$ (in the strongly convergent variant of $\Psi$) is chosen as $a_i(\tau_i(s+1)) \in BR_i^{\eta_{\tau_i(s+1)}}(p_i(\tau_i(s+1)-1))$. In order to translate this to the embedded FP-type process, define the following terms:
$\hat q^i_j(s) := q_j(\tau_i(s+1)-1),~
\hat q^i(s) := (q_1(\tau_i(s+1)-1),\ldots,q_n(\tau_i(s+1)-1))~
\hat p_i(s) := f_i^p(\hat q^i(s)),$
By construction, the $(s+1)$-th action of player $i$ in the embedded FP-type process is chosen as,
\begin{equation}
\label{embedded_BR}
\tilde a_i(s+1) \in BR_i^{\eta_{\tau_i(s+1)}}(\hat p_i(s)).
\end{equation}
In the embedded FP-type process, the term $\tilde q_j(s)$ may be thought of as the `true' empirical distribution of player $j$. The term $\hat q_j^i(s)$ may be thought of as the estimate which player $i$ maintains of $\tilde q_j(s)$, and the term $\hat q^i(s)$ (note the superscript) may be thought of as player $i$'s estimate of the joint empirical distribution $\tilde q(s)$ at the time of player $i$'s $(s+1)$-th best response. Finally, the term $\hat p_i(s)$ may be thought of as player $i$'s prediction of opponents next-stage strategy given $\hat q^i(s)$; in particular, note that---in the embedded FP-type process---player $i$ chooses their stage $(s+1)$ action \eqref{embedded_BR} as an asymptotic best response to $\hat p_i(s)$.

\subsection{Some Useful Properties}
\label{sec_useful_props}
Let
\begin{equation}
\Omega' := \{\omega: \lim_{t\rightarrow\infty} \frac{\ell_i(t)}{\sum_{k=1}^t \rho_i(t)} = 1,~ \forall i \}.
\end{equation}
By Lemma \ref{IR3} (see appendix), there holds $\mathbb{P}(\Omega')=1$. In proving Theorem \ref{theorem_general_result} we will restrict attention to (sample path) realizations in $\Omega'$.

Note that under assumption \textbf{A.\ref{rho_a2}}, there holds $\{\omega: \lim_{t\rightarrow\infty}\ell_i(t) = \infty,~ \forall i\}\supset \Omega'.$ By the equivalence $\{\omega: \lim_{t\rightarrow\infty}\ell_i(t) = \infty,~ \forall i\}=\{\omega:X_i(t)=1 \mbox{ infinitely often } \forall i\}$, there holds $\{\omega:X_i(t)=1 \mbox{ infinitely often } \forall i\}\supset \Omega'.$
Therefore, by the definitions of $\ell_i$ and $\tau_i$, there holds for any realization in $\Omega'$, $\lim_{t\rightarrow\infty} \ell_i(t) = \infty$, and
\vskip-15pt
\begin{align}
\label{tau_exist}
&\tau_i(s) <\infty, ~\forall s\in \mathbb{N},\\
\label{tau_lim}
&\lim\limits_{s\rightarrow\infty} \tau_i(s) = \infty.
\end{align}
\vskip-5pt

These properties will be useful in the proof of Theorem \ref{theorem_main_result}. In particular, the proof will frequently make reference to $\tilde q_i(s)$, or $\tilde a_i(s)$ for arbitrary $s\in \mathbb{N}$---the property \eqref{tau_exist} ensures that such terms are well defined for any $\omega \in \Omega'$.

Note also that for any realization in $\Omega'$, for $i\in N$ and $s\in \{1,2,\ldots\}$,
\vskip-15pt
\begin{equation}
\label{ell_tau_eq}
\ell_i(\tau_i(s)) = s,
\end{equation}
\vskip-15pt
\noindent and for $i\in N$ and $t\in \{1,2,\ldots\}$
\vskip-15pt
\begin{equation}
\label{X_i_implication}
X_i(t) = 1 \implies \tau_i(\ell_i(t)) = t.
\end{equation}
\vskip-5pt
\noindent Furthermore, note that $X_i(t) = 0$ implies that $\ell_i(t) = \ell_i(t-1)$ and $\bar H_i(t) = \bar H_i(t-1)$, and in particular,
\vskip-20pt
\begin{align}
\label{q_step_equality}
X_i(t) = 0 \implies q_i(t) = q_i(t-1).
\end{align}
\vskip-10pt
\noindent These facts are readily verified by conferring with the definitions of $\tau_i$, $\ell_i$, and $X_i$.

\subsection{Proof of Theorem \ref{theorem_general_result}}
\label{sec_proof_general_result}
\begin{proof}
Since $\mathbb{P}(\Omega') = 1$ it is sufficient to show that the desired result holds for any $\omega \in \Omega'$. Henceforth, we restrict attention to realizations $\omega \in \Omega'$, and for ease of notation suppress the term $\omega$ when referring to random variables.

As a first step, we wish to show that $\lim_{s\rightarrow\infty} d(\tilde \xi(s),~E) = 0$. We accomplish this by showing that there exists a sequence $\{\epsilon_s\}_{s\geq 1}$ such that $\lim_{s\rightarrow\infty}\epsilon_s = 0$ and $\tilde a_i(s+1) \in BR_i^{\epsilon_s}(\tilde p_i(s))$. By assumption \textbf{A.\ref{a_robustness}}, it will then follow that $\lim_{s\rightarrow\infty} d(\tilde \xi(s),~E) = 0$.

To that end, note that by Lemma \ref{lemma_BR_limit} (see appendix),
$\lim\limits_{s\rightarrow\infty} |U_i(a_i(\tau_i(s+1)),p_i(\tau_i(s+1)-1)) - v_i(p_i(\tau_i(s+1)-1))| = 0,~\forall i,$
or equivalently by the definitions of $\tilde a(s)$ and $\hat p_i(s)$ (see Section \ref{sec_additional_defs}),
\vskip-20pt
\begin{align}
\lim\limits_{s\rightarrow\infty} |U_i(\tilde a_i(s+1)),\hat p_i(s)) - v_i(\hat p_i(s))| = 0,~\forall i.
\label{thrm1_eq1}
\end{align}
\vskip-5pt
\noindent By Lemma \ref{IR0} (see appendix), $\lim_{s\rightarrow\infty} \|\hat q^i(s) - \tilde q(s)\| = 0$. By \textbf{A.\ref{a_prediction}}, it follows that $\lim_{s\rightarrow\infty} \|\hat p_i(s) - \tilde p_i(s)\| = 0$, which by the Lipschitz continuity of $U_i(\cdot)$ implies that
$\lim_{s\rightarrow\infty} | U_i(\alpha_i,\hat p_i(s)) - U_i(\alpha_i,\tilde p_i(s))| = 0,~ \forall \alpha_i \in A_i, \forall i,$
and
$\lim_{s\rightarrow\infty} |v_i(\hat p_i(s)) - v_i(\tilde p_i(s))| = 0, \forall i.$
Returning to \eqref{thrm1_eq1} we see that
$\lim\limits_{s\rightarrow\infty} |U_i(\tilde a_i(s+1)),\tilde p_i(s)) - v_i(\tilde p_i(s))| = 0, ~\forall i,$
i.e., there exists a sequence $\{\epsilon_s\}_{s\geq 1}$ such that $\epsilon_s \rightarrow 0$ and $\tilde a_i(s+1) \in BR_i^{\epsilon_s}(\tilde p_i(s))$. It follows by \textbf{A.\ref{a_robustness}} that
\vskip-20pt
\begin{equation}
\lim\limits_{s\rightarrow\infty} d(\tilde \xi(s),~E) = 0.
\label{theorem_main_result_eq3}
\end{equation}
\vskip-5pt

We now proceed to show that $\lim_{t\rightarrow\infty} d(\xi(t),~E) =0$. Let $\varepsilon > 0$ be given. By Lemma \ref{IR1} (see appendix) and assumption \textbf{A.\ref{a_xi}}, for each $i\in N$, there exists a random time $S_i > 0$ such that $\forall s\geq S_i$, $\|\xi(\tau_i(s)) - \tilde\xi(s)\| < \frac{\varepsilon}{2}$. Let $S^{'} = \max_i\{S_i\}$. By \eqref{theorem_main_result_eq3} there exists a random time $S^{''}$ such that $\forall s\geq S^{''}$, $d(\tilde \xi(s), ~E) < \frac{\varepsilon}{2}$. Let $S=\max\{S^{'},S^{''}\}$. Then
\vskip-15pt
\begin{equation}
d(\xi(\tau_i(s)),~E) < \varepsilon, ~\forall i,~ \forall s\geq S.
\label{thrm1_eq6}
\end{equation}
\vskip-5pt

Let $T = \max_{i}\{\tau_i(S)\}$. Note that for some $i$, $\xi(T) = \xi(\tau_i(S))$, and therefore by \eqref{thrm1_eq6},
\vskip-15pt
\begin{equation}
d(\xi(T),~E) < \varepsilon.
\label{thrm1_eq4}
\end{equation}
\vskip-5pt
Also note that for any $t_0>T$, it holds that $\ell_i(t_0) \geq S$ (since $\ell_i(\tau_i(S)) = S$, and $\ell_i(t)$ is non-decreasing in $t$), and moreover
\vskip-15pt
\begin{align}
X_i(t_0) = 1 \mbox{ for some } i & ~\implies ~q(t_0) = q(\tau_i(\ell_i(t_0))) \implies \xi(t_0) = \xi(\tau_i(\ell_i(t_0))),\\
X_i(t_0) = 0 \mbox{ for all $i$ } & ~\implies ~q(t_0) = q(t_0-1) \implies \xi(t_0) = \xi(t_0-1) ,
\label{thrm1_eq5}
\end{align}
\vskip-5pt
where the first implication holds with $ ~\mbox{ with } \ell_i(t_0) \geq S$. In the above, the first line follows from \eqref{X_i_implication}, and the second line follows from \eqref{q_step_equality}.
Consider $t\geq T$. If for some $i$, $X_i(t) = 1$, then by \eqref{thrm1_eq5} and \eqref{thrm1_eq6},
$d(\xi(t),~E) = d(\xi(\tau_i(\ell_i(t))),~E) < \varepsilon.$
Otherwise, if $X_i(t) = 0 ~\forall i$, then $\xi(t) = \xi(t-1)$.

Iterate this argument $m$ times until either (i) $X_i(t-m) = 1$ for some $i$, or (ii), $t-m = T$. In the case of (i),
$d(\xi(t),~E) = d(\xi(t-m),~E) = d(\xi(\tau_i(\ell_i(t-m))),~E) < \varepsilon,$
where the inequality again follows from \eqref{thrm1_eq6} and the fact that $t-m>T \implies \ell_i(t-m)\geq S$.
In the case of (ii),
$d(\xi(t),~E) = d(\xi(T),~E) < \varepsilon,$
where the inequality  follows from \eqref{thrm1_eq4}. Since $\varepsilon >0$ was chosen arbitrarily, it follows that
$\lim\limits_{t\rightarrow\infty} d(\xi(t),~E)=0.$

Finally, we show that $\lim_{t\rightarrow\infty} d(g(t),~E)=0$. Note that by \eqref{g_def_general}, $\|g_i(t) - \xi_i(t-1)\| \leq M_i\rho_i(t), ~\forall i,$ where $M_i:=\max_{p',p'' \in \Delta_i} \|p' - p''\|$ is a constant. Invoking assumption \textbf{A.\ref{rho_a1}} gives, $\lim\limits_{t\rightarrow\infty}\|g_i(t) - \xi_i(t-1)\|=0, ~\forall i.$
Combining this with the fact that $\lim\limits_{t\rightarrow\infty} d(\xi(t),~E)=0$ yields the desired result, $\lim_{t\rightarrow\infty}d(g(t),~E) =0$.
\end{proof}

\section{Applications of the General Result}
\label{sec_apps}
In this section we consider three different FP-type algorithms and study the strongly convergent variant of each. In each case, we prove strong convergence by showing that the FP-type algorithm fits the template of Theorem \ref{theorem_general_result}. Generally, the only non-trivial aspect of applying Theorem \ref{theorem_general_result} will be to show that \textbf{A.\ref{a_robustness}} is satisfied.

In Section \ref{sec_apps1} we consider classical FP. The fact that classical FP satisfies \textbf{A.\ref{a_robustness}} was shown by Leslie et al. \cite{leslie2006generalised}. In Section \ref{sec_apps2} we consider GWFP---a generalization of FP proposed in \cite{leslie2006generalised}. Again, the crucial step of showing that GWFP satisfies \textbf{A.\ref{a_robustness}} was shown in \cite{leslie2006generalised}. In Section \ref{sec_apps3} we consider a variant of FP termed ECFP. That ECFP satisfies \textbf{A.\ref{a_robustness}} was shown in \cite{swenson2015weakECFP}. We emphasize that each of these algorithms is known to achieve weak learning in the sense that $d(\xi(t),~E) \rightarrow 0$ as $t\rightarrow \infty$. Our contribution is to construct a variant where players also achieve learning in the strong sense that period-by-period mixed strategies also converge to equilibrium.

\subsection{Strong Convergence in Classical FP}
\label{sec_apps1}
We now prove Corollary \ref{theorem_main_result} using the general convergence result of Theorem \ref{theorem_general_result}.

\begin{proof}
Classical FP fits the template of an FP-type algorithm with the empirical distribution given by $q_i(t) = \frac{1}{t}\sum_{s=1}^ta_i(s)$, the functions $f_i^p$ and $f_i^\xi$ given by the identity function for each $i$, and the best response perturbation given by $\epsilon_t = 0,~\forall t$. To show that the strongly convergent variant of classical FP attains strong learning, it suffices to show that the assumptions of Theorem \ref{theorem_general_result} are met.

To that end, note that \textbf{A.\ref{rho_a1}}--\textbf{A.\ref{rho_a3}} are satisfied by assumption, and \textbf{A.\ref{a_eta}} is trivially satisfied (with $\eta_t = 0,~\forall t$). Furthermore, the empirical distribution sequence satisfies $\lim_{t\rightarrow\infty}\|q_i(t) - q_i(t-1)\| = 0$ (see Section \ref{sec_FP_example}), and hence \textbf{A.\ref{a_step_size_bound}} is satisfied. The functions $f_i^p$ and $f_i^\xi$ (each being the identity function) satisfy \textbf{A.\ref{a_prediction}}--\textbf{A.\ref{a_xi}}.
Therefore, it is sufficient to show that \textbf{A.\ref{a_robustness}} is satisfied. But, for zero-sum games, potential games, and generic $2\times m$ games this holds by \cite{leslie2006generalised}, Corollary 5.
\end{proof}

\subsection{Strong Convergence in Generalized Weakened FP}
\label{sec_apps2}

GWFP was introduced in Section \ref{sec_GWFP_intro}, where it was shown to fit the template of an FP-type algorithm.

Since, by definition, a GWFP process allows players to choose an $\epsilon_t$ sub-optimal best response with $\epsilon_t \rightarrow 0$, the following result (\cite{leslie2006generalised}, Corollary 5) guarantees a GWFP process satisfies \textbf{A.\ref{a_robustness}} in the noted classes of games.
\begin{theorem}
Any generalized weakened fictitious play process will converge to the set of Nash equilibria in two-player zero-sum games, potential games, and generic $2\times m$ games.
\label{thrm_wfp}
\end{theorem}

To clarify the precise meaning of the convergence stated above as it relates to the present work, we emphasize that Theorem \ref{thrm_wfp} implies that $\lim_{t\rightarrow\infty} d(q(t),~NE)=0$; i.e., the process converges weakly to equilibrium.

Let the strongly convergent variant of GWFP be constructed using the approach laid out in Section \ref{sec_FP_type_strong}. The following Corollary to Theorem \ref{theorem_general_result} states that the strongly convergent variant of a GWFP process will achieve strong learning.\footnote{It should be noted that classical FP may be seen as an instance of GWFP, and thus Corollary \ref{theorem_main_result} may in fact be deduced as a corollary to Corollary \ref{cor_gwfp}. However, for clarity and continuity of presentation, the results regarding classical FP have been presented separately.}
\begin{cor}
\label{cor_gwfp}
Let $\Gamma$ be a two-player zero-sum game, potential game, or generic $2\times m$ game. Let $\Psi$ be an instance of GWFP. If the strongly convergent variant of $\Psi$ satisfies \textbf{A.\ref{rho_a1}}--\textbf{A\ref{rho_a3}} and \textbf{A.\ref{a_eta}}, then it achieves strong learning in the sense that $\lim_{t\rightarrow\infty} d(g(t),~NE) = 0$.
\end{cor}
\begin{proof}
It is sufficient to show that the conditions of Theorem \ref{theorem_general_result} are met. Note that \textbf{A.\ref{rho_a1}}--\textbf{A.\ref{rho_a3}}, \textbf{A.\ref{a_eta}} hold by assumption. Furthermore, by definition, any GWFP process satisfies $\lim_{t\rightarrow\infty} \gamma(t)=0$, and hence satisfies \textbf{A.\ref{a_step_size_bound}}. The functions $f_i^p$ and $f_i^\xi$ are given by the identity function for each $i$, and hence \textbf{A.\ref{a_prediction}} and \textbf{A.\ref{a_xi}} hold.  Thus, it suffices to show that \textbf{A.\ref{a_robustness}} holds for the specified class of games---but, this follows from Theorem \ref{thrm_wfp}.
\end{proof}

\subsection{Strong Convergence in Empirical Centroid FP}
\label{sec_apps3}
ECFP was introduced in Sections \ref{sec_ECFP_intro} and \ref{sec_ecfp_mce_intro}. It
In order to study the asymptotic behavior of ECFP (in either of the above formats introduced in Sections \ref{sec_ECFP_intro} and \ref{sec_ecfp_mce_intro}) we make the following assumption regarding the structure of players' utility functions:
\begin{assumption}
The players' utility functions are identical and permutation invariant. That is, for any $i,j\in N$, $u_i(y) = u_j(y)$, and $u([y']_i,[y'']_j,y_{-(i,j)}) = u([y'']_i,[y']_j,y_{-(i,j)}),$
where, for any player $k\in N$, the notation $[y']_k$ indicates the action $y'\in Y_k$ being played by player $k$, and $y_{-(i,j)}$ denotes the set of actions being played by all players other than $i$ and $j$.
\label{a_perm_inv}
\end{assumption}

We note that, under this assumption, the sets $C$ and $MCE$ are nonempty \cite{swenson2012ECFP,swenson2013MCE}.
The following theorem (\cite{swenson2015weakECFP}, Theorem 1) specifies the manner in which players engaged in an ECFP process (weakly) learn elements of the sets $C$ and $MCE$.
\begin{theorem}
Let $\{a(t)\}_{t\geq 1}$ be an ECFP process.  \\
Assume $\Gamma$ is such that \textbf{A.\ref{a_ident_strat}} and \textbf{A.\ref{a_perm_inv}} hold. Then players learn equilibrium strategies in the sense that
(i) $\lim_{t\rightarrow \infty} d(\bar q^n(t),~C) = 0$, and (ii) $\lim_{t\rightarrow \infty} d(q(t),~MCE) = 0$.
\label{theorem_ecfp_weak}
\end{theorem}

Note that case (i) above corresponds to ECFP with the convergence map $f_i^\xi$ as given in Section \ref{sec_ECFP_intro}, and case (ii) corresponds to the convergence map $f_i^\xi$ given by the identity function (as in Section \ref{sec_ecfp_mce_intro}). Since, by definition, an ECFP process \eqref{ecfp_process} allows players to choose actions from the $\epsilon_t$-sub-optimal best response set with $\epsilon_t \rightarrow 0$, Theorem \ref{theorem_ecfp_weak} ensures that ECFP satisfies \textbf{A.\ref{a_robustness}}.

Let $\Psi$ be an instance of ECFP as presented in either Section \ref{sec_ECFP_intro} or Section \ref{sec_ecfp_mce_intro}, and let the strongly convergent variant of $\Psi$ be constructed using the approach laid out in Section \ref{sec_FP_type_strong}.
The following corollary to Theorem \ref{theorem_general_result} states that players engaged in the strongly convergent variant of an ECFP process learn elements of $C$ and $MCE$ in the strong sense that players' period-by-period strategies converge to equilibrium.
\begin{cor}
(i) Let $\Psi$ be an instance of ECFP with $f^{\xi}_i(q) = \frac{1}{n}\sum_jq_j,~\forall i$ and assume $\Gamma$ is such that \textbf{A.\ref{a_ident_strat}} and \textbf{A.\ref{a_perm_inv}} hold. If the strongly convergent variant of $\Psi$ satisfies \textbf{A.\ref{rho_a1}}--\textbf{A.\ref{rho_a3}} and \textbf{A.\ref{a_eta}}, then it achieves strong learning in the sense that $\lim_{t\rightarrow 0} d(g(t),~C)=0$.\\
(ii) Let $\Psi$ be an instance of ECFP with $f^{\xi}_i(q)$ given by the identity function for all $i$ and assume $\Gamma$ is such that \textbf{A.\ref{a_ident_strat}} and \textbf{A.\ref{a_perm_inv}} hold. If the strongly convergent variant of $\Psi$ satisfies \textbf{A.\ref{rho_a1}}--\textbf{A.\ref{rho_a3}} and \textbf{A.\ref{a_eta}}, then it achieves strong learning in the sense that $\lim_{t\rightarrow 0} d(g(t),~MCE)=0$.
\end{cor}
\begin{proof}
Cases (i) and (ii) differ only in terms of the function $f^{\xi}_i(t)$ and target equilibrium set $E$. However, in both cases the function $f^{\xi}_i$ satisfies \textbf{A.\ref{a_xi}}. It suffices to show the remaining conditions of Theorem \ref{theorem_general_result} are satisfied. Henceforth we treat cases (i) and (ii) equivalently.

Note that \textbf{A.\ref{rho_a1}}--\textbf{A.\ref{rho_a3}} and \textbf{A.\ref{a_eta}} hold by assumption. The empirical distribution sequence satisfies $\|q_i(t) - q_i(t-1)\| \leq \frac{M_i}{t} \rightarrow 0 \mbox{ as } t\rightarrow \infty$, where $M_i := \sup_{p',p'' \in \Delta_i} \|p' - p''\|$, and hence \textbf{A.\ref{a_step_size_bound}} is satisfied. Note that the function $f_i^p(q) = \frac{1}{n}\sum_j q_j$ satisfies \textbf{A.\ref{a_prediction}}.  Finally, Theorem \ref{theorem_ecfp_weak} shows that \textbf{A.\ref{a_robustness}} is satisfied.
\end{proof}

\section{Conclusions}
\label{sec_conclusion}
An algorithm is said to achieve weak learning if players learn an equilibrium strategy in an abstract sense (see Section \ref{sec_prelims}), but period-by-period strategies do not necessarily converge to equilibrium. An algorithm is said to achieve strong learning if (additionally) players' period-by-period strategies converge to equilibrium. Weak learning may be thought of as a form of learning where players \emph{learn} a strategy in some abstract sense, but never begin to implement the strategy they are learning. On the other hand, in strong learning, not only do players \emph{learn} a strategy, but they also physically implement the learned strategy through the course of the learning process.

Fictitious Play (FP) and its variants are known to exhibit weak learning but not necessarily strong learning. An approach was presented for taking a general FP-type algorithm that achieves weak learning, and constructing from it a strongly convergent variant of the algorithm. General convergence results were proved and used to construct a strongly convergent variant of several example FP-type processes.

In order to apply the convergence results proved in this paper, it is necessary to ensure a candidate algorithm meets \textbf{A.\ref{a_robustness}} (the other necessary assumptions are relatively trivial to verify). An interesting future research direction might be to investigate other FP-type algorithms (e.g., \cite{Arslan04,Shamma03}) and verify whether they meet the assumptions sufficient for construction of a strongly convergent variant.

\section*{Appendix}
{\small
\subsection{Some Useful Inequalities}
\label{sec_IR}
We consider some useful inequalities related to the strongly convergent variant of an FP-type algorithm. We restrict attention to realizations $\omega \in \Omega'$. Let $\{q_i(t)\}_{t\geq 1}$ be given by \eqref{qt_general_strong_update}.
By \textbf{A.\ref{a_step_size_bound}} there exists a sequence $\gamma(t)$ such that
$\lim\limits_{t\rightarrow\infty} \gamma(t) = 0,$
and for each $i\in N$,
\vskip-5pt
\begin{equation}
\|q_i(t+1) - q_i(t)\| \leq M_i\gamma(\ell_i(t)),
\label{qt_bound}
\end{equation}
where $M_i:= \sup_{q',q'' \in \Delta_i}\|q' - q''\|$.
Similarly, there holds for any integer $s>0$,
\begin{equation}
\|\tilde q(s+1) - \tilde q(s)\| \leq M\gamma(s),
\label{qs_bound1}
\end{equation}
where $M:= \sup_{q',q'' \in \Delta^n}\|q' - q''\|$.
More generally, for any integers $s_1,s_2>0$, if \textbf{A.\ref{a_step_size_bound}} holds then,
\vskip-20pt
\begin{align}
\|\tilde q(s_1) - \tilde q(s_2)\| & \leq M\sum\limits_{s=\min\{s_1,s_2\}}^{\max\{s_1,s_2\}-1} \gamma(s)\leq |s_1 - s_2|B,
\label{qs_bound2}
\end{align}
\vskip-5pt
\noindent where $0<B<\infty$ is such that $\sup_t \gamma(t) \leq B/M$.

\subsection{Intermediate Results}
$~$\\
\vskip-5pt
\begin{lemma}
\label{lemma_BR_limit}
Let $\tau_i(s)$ be defined as in section \ref{sec_additional_defs}, and assume \textbf{A.\ref{a_eta}} holds. Then for any realization in $\Omega'$ there holds,
$\lim\limits_{s\rightarrow\infty} |U_i(a_i(\tau_i(s)),p_i(\tau_i(s)-1)) - v_i(p_i(\tau_i(s)-1))| = 0,~\forall i.$
\end{lemma}
\begin{proof}
Let $s\in \mathbb{N}$. Note that by definition $\tau_i(s) := \inf\{t:\ell_i(t) = s\}$ and $\ell_i(t) := \sum_{k=1}^t X_i(k)$, thus $X_i(\tau_i(s)) = 1$. By \eqref{action_rule1} this implies
$a_i(\tau_i(s)) = b_i(\tau_i(s)) \in BR_i^{\eta_{\tau_i(s)}}(p_i(\tau_i(s)-1)),$
which implies
$|U_i(a_i(\tau_i(s)),p_i(\tau_i(s)-1)) - v_i(p_i(\tau_i(s)-1))| \leq \eta_{\tau_i(s)}.$
By \textbf{A.\ref{a_eta}}, $\eta_t \rightarrow 0$ as $t\rightarrow \infty$, and moreover, by \eqref{tau_lim}, $\tau_i(s) \rightarrow \infty$ as $s\rightarrow \infty$. Thus $\eta_{\tau_i(s)} \rightarrow 0$ as $s\rightarrow \infty$, and the claim holds.
\end{proof}

\begin{lemma}
Let $i,j\in N$, let $\tau_i(s)$ and $\tilde q_j(s)$ be defined as in Section \ref{sec_additional_defs}, and assume \textbf{A.\ref{rho_a2}}--\textbf{A.\ref{rho_a3}} hold. Then for any realization in $\Omega'$,
$\lim_{s\rightarrow\infty} \|  q_j(\tau_i(s)) - \tilde q_j(s)\| = 0.$
\label{IR1}
\end{lemma}
\begin{proof}
Note that for any $t\in \mathbb{N}$,
$q_j(t) = q_j(\tau_j(\ell_j(t))) =  \tilde q_j(\ell_j(t)),$
where the first equality follows from Lemma \ref{IR6}, and the second equality follows from the definition of $\tilde q_i(s)$. Hence,
\vskip-15pt
\begin{align}
\|q_j(\tau_i(s)) - \tilde q_j(s)\| = \|\tilde q_j(\ell_j(\tau_i(s))) - \tilde q_j(s)\|
& = \|\tilde q_j(\ell_j(\tau_i(s))) - \tilde q_j(\ell_i(\tau_i(s)))\|\\
& \leq \vert\ell_j(\tau_i(s)) - \ell_i(\tau_i(s))\vert B,
\label{IR1_eq1}
\end{align}
\vskip-5pt
where the first equality follows from the previous statement, and the second equality follows from the fact that $\ell_i(\tau_i(s)) = s$ (see \eqref{ell_tau_eq}), and the final inequality follows from \eqref{qs_bound2}. Thus, it suffices to show that
\vskip-10pt
\begin{equation}
\lim\limits_{s\rightarrow\infty} \vert\ell_j(\tau_i(s)) - \ell_i(\tau_i(s))\vert = 0.
\label{IR1_eq2}
\end{equation}
For convenience in notation let $h_i(t) := \sum_{m=1}^t \rho_i(m)$. By Lemma \ref{IR3} and the definition of $\Omega'$ there holds for any $k\in N$,
$\lim_{t\rightarrow\infty} \frac{ \ell_k(t) }{h_k(t)} = 1.$
By assumption \textbf{A.\ref{rho_a3}}, for any $k\in N$, $\lim_{t\rightarrow\infty} \left(h_k(t) / (h_i(t)\right) = 1.$ Hence, for any $k\in N$,
\begin{align}
\lim\limits_{t\rightarrow\infty} \frac{\ell_k(t)}{h_i(t)} = \lim\limits_{t\rightarrow\infty} \frac{\ell_k(t)}{h_k(t)} \frac{h_k(t)}{h_i(t)}
=1.
\label{IR1_eq4}
\end{align}
Returning attention to \eqref{IR1_eq2} and recalling that by \eqref{tau_lim}, $\lim_{s\rightarrow\infty}\tau_i(s)=\infty$ on $\Omega'$, we have,
\vskip-10pt
\begin{align}
\limsup\limits_{s\rightarrow\infty}\left| \ell_j(\tau_i(s)) - \ell_i(\tau_i(s))\right|
\leq \limsup\limits_{t\rightarrow\infty} \left|\ell_j(t) - \ell_i(t) \right|
& = \limsup\limits_{t\rightarrow\infty} \left| \frac{\ell_j(t)}{h_i(t)} h_i(t) - \frac{\ell_i(t)}{h_i(t)}h_i(t)\right|\\
& = \limsup\limits_{t\rightarrow\infty} \left|h_i(t) - h_i(t)\right| = 0,
\label{IR1_eq3}
\end{align}
\vskip-5pt
\noindent where the transition to the last line follows from application of \eqref{IR1_eq4}.
Thus, \eqref{IR1_eq2} is verified, and the desired result holds.
\end{proof}

\begin{lemma}
Let $i,j\in N$, let $\hat q_j^i(s)$ and $\tilde q_j(s)$ be defined as in Section \ref{sec_additional_defs}, and assume \textbf{A.\ref{rho_a2}}--\textbf{A.\ref{rho_a3}} hold. Then for any realization in $\Omega'$ there holds
$\lim_{s\rightarrow\infty} \|  \hat q_j^i (s) - \tilde q_j(s)\| = 0.$
\label{IR0}
\end{lemma}
\begin{proof}
Recall that by definition, $\hat q_j^i(s) = q_j(\tau_i(s+1)-1)$; our objective then is to show that
$\lim_{s\rightarrow\infty} \|  q_j (\tau_i(s+1)-1) - \tilde q_j(s)\| = 0.$
By Lemma \ref{IR1},
$\lim\limits_{s\rightarrow\infty} \|q_j(\tau_i(s)) - \tilde q_j(s)\| = 0.$
By \eqref{qs_bound1} and \textbf{A.\ref{a_step_size_bound}} there holds,
$\lim\limits_{s\rightarrow\infty} \|\tilde q_j(s+1) - \tilde q_j(s)\| = 0.$
Combining this with the previous statement,
\vskip-10pt
\begin{equation}
\lim\limits_{s\rightarrow\infty} \|q_j(\tau_i(s+1)) - \tilde q_j(s)\|=0.
\label{IR0_eq2}
\end{equation}
\vskip-5pt
\noindent Recalling \eqref{qt_bound}, there holds,
\vskip-15pt
\begin{align}
\limsup\limits_{s\rightarrow\infty} \|q_j(\tau_i(s+1)-1) - q_j(\tau_i(s+1))\| \leq \limsup\limits_{s\rightarrow\infty}M_{j} \gamma(\ell_j(\tau_i(s+1))) = 0,
\label{IR0_eq3}
\end{align}
\vskip-10pt
\noindent where the equality holds since $\lim_{s\rightarrow\infty} \ell_j(\tau_i(s)) = \infty$ on $\Omega'$, and by \textbf{A.\ref{a_step_size_bound}}, $\lim_{s\rightarrow\infty} \gamma(s) = 0$.

Consider now the quantity of interest,{\small
\vskip-15pt
\begin{align}
\|q_j(\tau_i(s+1)-1) - \tilde q_j(s)\| \leq & \|q_j(\tau_i(s+1)-1) - q_j(\tau_i(s+1))\| + \|q_j(\tau_i(s+1)) - \tilde q_j(s)\|.
\end{align}
}
\vskip-15pt
\noindent The first term on the right hand side (RHS) goes to zero by \eqref{IR0_eq3} and the second term on the RHS goes to zero by \eqref{IR0_eq2}. Thus,
$\lim\limits_{s\rightarrow\infty} \|q_j(\tau_i(s+1)-1) - \tilde q_j(s)\| = 0,$ and the claim holds.
\end{proof}

\begin{lemma}
Let $i\in N$, let $q_i(\cdot)$ be as defined in \eqref{qt_general_strong_update}, let $\ell_i(\cdot)$ be as defined in \eqref{ell_def2}, and let $\tau_i(\cdot)$ be as defined in \eqref{tau_def}. Then for every realization in $\Omega'$ and any $t\in\{1,2,\ldots\}$ there holds $q_i(\tau_i(\ell_i(t))) = q_i(t).$
\label{IR6}
\end{lemma}
\begin{proof}
Let
$t_0 := \tau_i(\ell_i(t)) = \inf\{t':\ell_i(t') = \ell_i(t)\},$ where the second equality follows from the definition of $\tau_i(\cdot)$. Note that $t_0 \leq t$ and by definition of $t_0$, there holds $\tau_i(\ell_i(t_0)) = t_0$, and hence
$ q_i(\tau_i(\ell_i(t_0)))=q_i(t_0).$
Furthermore, by the definition of $t_0$, for $t_0  \leq t' \leq t$, there holds $\ell_i(t) = \ell_i(t') = \ell_i(t_0)$, and hence
$\tau_i(\ell_i(t)) = \tau_i(\ell_i(t_0)).$
Moreover, the fact that $\ell_i(t) = \ell_i(t') = \ell_i(t_0)$ implies by definition of $\ell_i(\cdot)$ that $X_i(t') = 0$ for $t_0 < t' \leq t$ (if such a $t'$ exists). Thus, by \eqref{q_step_equality} there holds $q_i(t) = q_i(t') = q_i(t_0)$ for $t_0 \leq t'\leq t$, and in particular $q_i(t) = q_i(t_0).$
Combining this with the facts that $q_i(\tau_i(\ell_i(t_0)))=q_i(t_0)$ and $\tau_i(\ell_i(t)) = \tau_i(\ell_i(t_0))$ yields the desired result.
\end{proof}

\begin{lemma}
\label{q_tilde_lemma}
Let $\Psi = (\{f^q_i(\cdot,t)\}_{t\geq 1},f_i^p,f^\xi_i)_{i\in N}$ be an FP-type algorithm, and let the strongly convergent variant of $\Psi$ be constructed as in Section \ref{sec_FP_type_strong}.
Let $\tilde a(s)$, $\tilde H(s)$, and $\tilde q_i(s)$ be as defined in Section \ref{sec_additional_defs}. Then for every realization in $\Omega'$, and for $s\geq 1$,
$\tilde q_i(s) = f_i^q(\tilde H(s),s).$
\end{lemma}
\begin{proof}
For $s\geq 1$, note that
$\tilde q_i(s) = q_i(\tau_i(s)) = f_i^q(\bar H_i(\tau_i(s)),\ell_i(\tau_i(s))) = f_i^q(\tilde H(s),s)$, where the first equality follows from the definition of $\tilde q_i(s)$ in Section \ref{sec_additional_defs}, the second follows from \textbf{A.\ref{a_general_q}}, and the third follows from the definition of $\tilde H_i(s)$ in Section \ref{sec_additional_defs} and \eqref{ell_tau_eq}.
\end{proof}

\begin{lemma}
Let $\{X(t)\}_{t\geq 1}$ be $0 - 1$ Bernoulli random variables, let $\ell(t) := \sum_{k=1}^t X(k)$ be the associated counting process, let $\mathcal{G}_t := \sigma(\{X(k)\}_{k=1}^t)$, and let $\rho(t) = \mathbb{P}(X(t) = 1|\mathcal{G}_{t-1})$.  Assume $\sum_{t\geq 1} \rho(t) = \infty$. Then there holds, $\lim\limits_{t\rightarrow\infty} \left(\ell(t)\right)/\left(\sum_{k=1}^t \rho(t)\right) = 1, ~\mbox{ a.s.}$
\label{IR3}
\end{lemma}
\begin{proof}
The result follows via Levi's extension of the Borel-Cantelli Lemmas, \cite{williams_book} p.124.
\end{proof}
}

\bibliographystyle{unsrt}
\bibliography{myRefs}

\end{document}